\documentclass[onefignum,onetabnum]{siamart190516}

\usepackage{lipsum}
\usepackage{amsfonts}
\usepackage{graphicx}
\usepackage{epstopdf}
\usepackage{algorithmic}
\ifpdf
  \DeclareGraphicsExtensions{.eps,.pdf,.png,.jpg}
\else
  \DeclareGraphicsExtensions{.eps}
\fi

\newsiamremark{remark}{Remark}
\newsiamremark{hypothesis}{Hypothesis}
\crefname{hypothesis}{Hypothesis}{Hypotheses}
\newsiamthm{claim}{Claim}

\usepackage{amsmath,amssymb,mathtools,setspace,dsfont,color}
\usepackage{xspace,bm}
\usepackage{todonotes}
\usepackage{algorithmic} 
\usepackage[resetcount,vlined,boxruled,boxed,algo2e]{algorithm2e}
\usepackage{ulem}

\def\cG{\mathcal{G}\xspace}

\def\cV{\mathcal{V}\xspace}

\def\CC{\mathbb{C}\xspace}
\def\RR{\mathbb{R}\xspace}
\def\ZZ{\mathbb{Z}\xspace}
\def\QQ{\mathbb{Q}\xspace}
\def\NN{\mathbb{N}\xspace}

\def\MM{\mathbb{M}\xspace}
\def\VV{\mathbb{V}\xspace}

\def\x{\bm{x}\xspace}

\def\x{\overline{x}\xspace}

\newcommand{\Elim}{\mathtt{Elim}\xspace}

\DeclarePairedDelimiter\abs{\lvert}{\rvert}
\DeclarePairedDelimiter\norm{\lVert}{\rVert}

\newtheorem{problem}[theorem]{Problem}

\headers{The Multivariate Schwartz-Zippel Lemma}{M.~L.~Do\u{g}an, A.A.~Erg\"ur, J.~D.~Mundo, and E.~Tsigaridas}

\title{The Multivariate Schwartz-Zippel Lemma
  \thanks{Submitted to the editors 2020.
    \funding{
      LD is supported by ERC under the European's Horizon 2020
      research and innovation programme (grant agreement no. 787840).
      AE is supported by Einstein Foundation, Berlin, NSF REU Grant
      DMS-1460766, and NSF CCF 2110075.
      JM is supported by NSF REU Grant DMS-1460766.
      ET is supported by the ANR JCJC GALOP (ANR-17-CE40-0009), the
      PGMO grant ALMA, and the PHC GRAPE.}}}

\author{
  M. Levent Do\u{g}an\thanks{Technische Universit\"at Berlin, Institut
    f\"ur Mathematik, Strasse des 17. Juni 136, 10623, Berlin,
    Germany
    (\email{dogan.mlevent@gmail.com})}
  \and
  Alperen A. Erg\"ur\thanks{The University of Texas at San Antonio, One UTSA Circle, San Antonio, TX, 78249
    (\email{alperen.ergur@utsa.edu})}
  \and
  Jake D. Mundo\thanks{Brown University, Division of Applied Mathematics,
    Providence, RI
    (\email{jake\_mundo@brown.edu})}
  \and
  Elias~Tsigaridas\thanks{Inria Paris and Institut de
  Math{\'e}matiques de Jussieu-Paris Rive Gauche, Sorbonne
  Universit\'e and Paris Universit\'e,
  France (\email{elias.tsigaridas@inria.fr})}
 }

\usepackage{amsopn}

\ifpdf
\hypersetup{
  pdftitle={Our title},
  pdfauthor={Us}
}
\fi

\begin{document}

\maketitle

\begin{abstract}
Motivated by applications in combinatorial geometry, we consider the following question:  Let $\lambda=(\lambda_1,\lambda_2,\ldots,\lambda_m)$ be an $m$-partition of a positive integer $n$, $S_i \subseteq \mathbb{C}^{\lambda_i}$ be finite sets, and let $S:=S_1 \times S_2 \times \ldots \times S_m \subset \mathbb{C}^n$ be the multi-grid defined by $S_i$. Suppose $p$ is an $n$-variate degree $d$ polynomial. How many zeros does $p$ have on $S$?
  
We first develop a multivariate generalization of the Combinatorial Nullstellensatz that certifies the existence of a point $t \in S$
so that $p(t) \neq 0$. Then, we show that a natural multivariate generalization of the DeMillo-Lipton-Schwartz-Zippel lemma holds, except for a special family of polynomials that we call $\lambda$-reducible.  This yields a simultaneous generalization of Szemer\'edi-Trotter theorem and Schwartz-Zippel lemma into higher dimensions, and has applications in incidence geometry. Finally, we develop a symbolic algorithm  that identifies certain $\lambda$-reducible polynomials. More precisely, our symbolic algorithm detects polynomials that include a cartesian product of hypersurfaces in their zero set. It is likely that using Chow forms the algorithm can be  generalized to handle arbitrary $\lambda$-reducible polynomials, which we leave as an open problem.
\end{abstract}



\section{Introduction}
\label{sec:intro}
Incidence geometry studies properties of arrangements of geometric
objects such as lines, points, and hypersurfaces.  The focus is on
understanding extremal geometric configurations. For example, if $P$ is collection of points in the real plane with cardinality $m$, and $L$ is a collection of lines with cardinality $n$, one is interested in understanding what could be the maximal number of incidences between the elements of $P$ and $L$ in terms of $m$ and $n$.

Incidence geometry is a fundamental and relatively old field in
combinatorics. Yet, its quests in the last two decades have unraveled surprising connections between theoretical computer science, harmonic analysis, and number theory vindicating long standing conjectures \cite{guth-katz,dvir2009size,guth2010algebraic}. Most of these recent breakthroughs in incidence geometry rely on creatively using (the zero sets of) polynomials and classical algebraic geometry.  This circle of ideas already formed an emerging field in combinatorics called ``the polynomial method'' \cite{dvir-survey,tao,guth2016polynomial,Sheffer-book-09}.  

One of the early achievements of the polynomial method is Alon's Combinatorial Nullstellensatz \cite{alon}, which reads as follows. 
\begin{theorem}[Combinatorial Nullstellensatz]
  \label{thm:alon-cn}
  Let $\mathbb{F}$ be a field and let $p \in \mathbb{F}[x_1,x_2, \ldots, x_n]$ be a polynomial of degree 
  $\deg(p)=\sum_{i}^n t_i$ for some positive integers $t_i$; further assume that the coefficient of $\prod_{i=1}^{n} x_i^{t_i}$ in $p$ is non-zero. Let $S_i \subseteq \mathbb{F}$, for $i \in [n]$, be finite sets with $\abs{S_i} > t_i$ and 
  $S:=S_1 \times S_2 \times \ldots S_n \subseteq \mathbb{F}^n$. Then, there exists $t \in S$ such that $ p(t) \neq 0$. 
\end{theorem}
The deep impact of Alon's result in combinatorics hinges on the fact that one can translate a diverse set of questions in structural graph theory, additive number theory, and Diophantine equations into the problem of locating zeros of a polynomial on a grid of points.

In the opposite direction to Alon's theorem, one can also try to use a known grid of points to understand  an unknown polynomial.  This direction represents a subject of extensive research in symbolic computation where the basic result is DeMillo-Lipton-Schwartz-Zippel (DLSZ) lemma~\cite{lipton,saxena,sharir}.
\begin{theorem}[DeMillo-Lipton-Schwartz-Zippel Lemma] \label{thm:uni-SZ}
Let $S \subseteq \mathbb{F}$ be a finite set where $\mathbb{F}$ is a field and let $p \in \mathbb{F}[x_1,x_2, \ldots,x_n]$ be a
polynomial of degree $d$. Assume $\abs{S} > d$ and consider 
$S^n = S \times S \times \ldots \times S$.
If $Z(p)$ is the zero set of $p$, then we have
  \[ \abs{Z(p) \cap S^n} \leq d \abs{S}^{n-1}, \]
where $\abs{.}$ denotes the cardinality of a set.
\end{theorem}

A common view of these two classical results would be that Combinatorial Nullstellensatz certifies if a given polynomial $p$ entirely vanishes on a grid, where  DLSZ lemma provides quantitative estimates on the number of zeros of $p$ on the grid. 

Mojarrad, Pham, Valculescu, and de Zeeuw  further noticed that the question of estimating the number of zeros of a $4$-variate polynomial on a finite subset of $\mathbb{C}^2 \times \mathbb{C}^2$ is equivalent to classical questions in incidence geometry \cite{dezeeuw}. Motivated by Mojarrad, Pham, Valculescu, and de Zeeuw's observation, we consider the following questions: Let $\lambda_i \in \NN$, for $1 \leq i \leq m$, be a partition of $n$, that is $\sum_{i=1}^m \lambda_i= n$, and let $S_i \subseteq \mathbb{C}^{\lambda_i}$. Consider the multi-grid $S$ defined as
\[ S := S_1 \times S_2 \times \cdots \times S_m  .\]
For a given $n$-variate polynomial $p$ of  degree $d$ can we prove that there is a $t \in S$ with $p(t) \neq 0$? Can we go one step further and prove an upper bound for the number of zeros of $p$ on $S$?

We prove multivariate generalizations of Combinatorial Nullstellensatz and DLSZ Lemma which answers both of these questions. These results have  immediate consequences in incidence geometry, some of which  we collect in Section~\ref{sec:appl-CG}.



\subsection{Summary of Our Results}
First, we present a generalization of Alon's Combinatorial Nullstellensatz (\Cref{combnull}) on a multi-grid. Then, we develop a generalization of DeMillo-Lipton-Schwartz-Zippel Lemma (\Cref{thm-SZ}) to the multivariate setting. Here matters are considerably more complicated as certain multi-grids and polynomials do not allow any nontrivial bound.
We assume that the structure of the multi-grid is a priori not known, and we seek polynomials that are compatible with any multi-grid. We call this family of polynomials  $\lambda$-irreducible polynomials (\Cref{def:lambda-irr}). The proofs are in \Cref{sec:proofs}.

Our algebraic toolbox allows us to obtain bounds for a variety of geometric configurations in a unified way (\Cref{sec:appl-CG}).
We derive some corollaries such a complex version of Szemer\'edi-Trotter theorem (\Cref{cor:ST-in-C}), that are known to be sharp, to demonstrate sharpness of our results.

We present an algorithm (\Cref{sec:algo}) to recognize certain
$\lambda$-reducible polynomials: it detects if there are hypersurfaces $V_i \subset \mathbb{C}^{\lambda_i}$ that satisfy $V_1 \times V_2 \times \ldots \times V_m \subset Z(p)$.  The construction is based on some classical, yet powerful, tools from computational algebra such as multivariate resultants and regular sequences which we introduce in (\Cref{sec:-alg-prelim}). The algorithm has the potential to be generalized to handle arbitrary $\lambda$-reducible polynomials using Chow forms, but we leave this as an open problem.

\paragraph{Comparison with previous work}
\label{sec:previous-work}
There have been numerous remarkable articles that focus on incidences
between a collection of real algebraic sets and a set of points
under certain tameness assumptions. We humbly provide a sample
\cite{tao-solymosi,pach-sharir,sheffer,adam-semi} but we warn the
reader that they only represent the tip of an iceberg.
These articles pose natural assumptions on the input data and derive sharp estimates. However, in most of the cases, these combinatorial assumptions on the input data are not formalized as a checkable condition but assumed to be granted.
Our contribution consists in identifying a workable assumption on the input ($\lambda$-irreducibility, see \cref{def:lambda-irr}) that we equip with an algorithm to detect certain $\lambda$-irreducible polynomials.

The bounds of our main result (Thm.~\ref{thm-SZ}) seem to be sharp 
at this level of generality except some loss in the exponents of $d$; this can be seen from sharp bounds presented in various corollaries in \cref{sec:appl-CG}. However, unlike the mentioned results in the literature, our estimates do not improve with extra assumptions on the data.

\subsection{Multivariate Combinatorial Nullstellensatz}
\label{sec:mCN}
To present a  multivariate generalization of the Combinatorial
Nullstellensatz, we need to introduce
the algebraic degree of a finite set.
\begin{definition}[Algebraic Degree of a Finite Set]
  Let $\mathbb{F}$ be a field and let $S \subset \mathbb{F}^n$ be a
  finite set of points. Let $I(S) \subset \mathbb{F}[x_1,\ldots,x_n]$
  be the ideal of polynomials vanishing on $S$. We define
  \[ \deg(S) := \min_{p \in I(S)} \deg(p) \]
  \noindent to be the algebraic degree of $S$.
\end{definition}
For the univariate case, that is when $S \subset \mathbb{F}$, it holds that
$\deg(S)=\abs{S}$. This was one of the key observations in Alon's
celebrated Combinatorial Nullstellensatz. However, for $n \geq 2$ one
can have arbitrarily large sets of degree one in $\mathbb{F}^n$: just
consider many points sampled from a hyperplane. The only general
relation between the size and the degree of a set
$S \subset \mathbb{F}^n$ seems to be the  following inequality 
\[ \abs{S} \geq \binom{\deg(S) -1 +n}{n} \]
This inequality can proved by linear algebra as follows:  Let the vector space of polynomials of degree at most $\deg(S)-1$ be denoted by $P_{\deg(S)-1}$, and consider the linear map $L_S : P_{\deg(S)-1} \rightarrow \mathbb{F}^{\abs{S}}$ that is defined by listing point-wise evaluations of a polynomial on the set $S$.
Then, if $\abs{S} < \dim(P_{\deg(S)-1})$, the map $L_S$ has non-trivial kernel, i.e. there exists a polynomial vanishing to the entire set $S$.

Let $\lambda=(\lambda_1, \lambda_2, \ldots, \lambda_m)$ be an $m$-partition of $n$, i.e., $n=\sum_i \lambda_i$. We will consider multi-grids with partition structure determined by $\lambda$.
So, it makes sense to use the following convention;
$\bar{x}_1=(x_1,x_2, \ldots, x_{\lambda_1})$,
$\bar{x}_2=(x_{\lambda_1+1}, x_{\lambda_1+2}, \ldots,
x_{\lambda_1+\lambda_2} )$, and so on. Moreover, we will use $\deg_i(p)$ to denote the degree of $p$ with respect to the variables $\bar{x}_i$. It holds that $\deg_{i}(p) \leq \deg(p)$ and $\deg(p) \leq \sum_i \deg_i(p)$.  

\begin{theorem}[Multivariate  Combinatorial Nullstellensatz]
  \label{combnull}
  Let $\mathbb{F}$ be a field, $\lambda=(\lambda_1, \lambda_2, \ldots, \lambda_m)$ be an $m$-partition of $n$, and let $p(x) \in \mathbb{F}[x_1,x_2,\ldots,x_n]$ be a polynomial with $\deg_i(p)=d_i$. Furthermore, assume that the coefficient of $x^{\alpha}$ in $p(x)$, that satisfies
  $ \deg_i(x^{\alpha})=d_i$ for all $i \in [m]$, is not zero. Let
  $S_i \subset \mathbb{F}^{\lambda_i}$ be finite sets and consider the
  multi-grid
  $S := S_1 \times S_2 \times \ldots \times S_{m} \subset \mathbb{F}^n$.
  If $\deg(S_i) > d_i $ for all
  $i \in [m]$, then  there exists a $t \in S$ such
  that $ p(t) \neq 0$.
\end{theorem}
The proof of the theorem appears in \cref{sec:proof-CN}.

\subsection{Multivariate Schwartz-Zippel Lemma}
\label{sec:mSZ}

In the same way we use Alon's Nullstellensatz to understand the structure of a
grid created from a number of finite sets $S_i \subseteq \mathbb{F}$, we can use its
multivariate extension supported by Theorem \ref{combnull} to certify
$S \not \subset Z(p)$ for a multi-grid.
However, for  applications in incidence geometry one needs a quantitative statement, like DeMillo-Lipton-Schwartz-Zippel
lemma, for multi-grids. We provide such a quantitative statement in
Theorem~\ref{thm-SZ}, but first we would like to consider some examples.

Consider
$g_1, g_2 \in \mathbb{C}[x_1,x_2,x_3]$, $g_3,g_4 \in \mathbb{C}[x_4,x_5,x_6]$,
and  $p=g_1 h_1 + g_2 h_2 + g_3 h_3 + g_4 h_4$, where
$h_i \in \mathbb{C}[x_1,x_2,\ldots,x_6]$. If the polynomials $g_1$ and $g_2$ are
generic, then the intersection $Z(g_1) \cap Z(g_2)$ is a one dimensional
variety; the same holds for $Z(g_3) \cap Z(g_4)$.
We have the following inclusion:
\[ ( Z(g_1) \cap Z(g_2) ) \times ( Z(g_3) \cap Z(g_4) ) \subseteq Z(p)   .\] 
Therefore, for any sets $S_1 \subset Z(g_1) \cap Z(g_2)$ and $S_2 \subset Z(g_3) \cap Z(g_4)$, we have $S_1 \times S_2 \subset Z(p)$. More generally, we can consider two positive dimensional varieties
$V_1 , V_2 \subseteq \mathbb{C}^n$ and then take an ideal sum in $\mathbb{C}[x_1,\ldots, x_n, x_{n+1}, \ldots, x_{2n}]$: $I := I(V_1) + I(V_2)$. Then, any polynomial $f \in I$ vanishes on $V_1 \times V_2$.

These examples indicate that to have a quantitative statement on
$\abs{Z(p) \cap S}$ one has to assume certain compatibility conditions between  $p$ and $S$.  Since we assume the structure of the multi-grid $S$ is a priori not known, we take it as a Gordian knot and seek polynomials that are compatible with any grid. We call such polynomials $\lambda$-irreducible.
\begin{definition}[$\lambda$-irreducible algebraic sets]
  \label{def:lambda-irr}
  Let $\lambda$ be an $m$-partition of $n$ and let $V \subseteq \mathbb{C}^n$ be
  an algebraic set. We say $V$ is $\lambda$-reducible if there exist positive
  dimensional varieties $V_i \subseteq \mathbb{C}^{\lambda_i}$, for $i \in [m]$,
  such that
  \[ V_1 \times V_2 \times \ldots \times V_m \subseteq V . \]
  We call $V$ a $\lambda$-irreducible algebraic set otherwise. If $V$ is a
  hypersurface defined by a polynomial $p$, then we say $p$ is
  $\lambda$-reducible (resp. $\lambda$-irreducible).
\end{definition}
Mojarrad, Pham, Valculescu and de Zeeuw \cite{dezeeuw} studied the problem
for the special case where $\lambda=(2,2)$. Note that a positive dimensional variety embedded in $\mathbb{C}^2$ is either the entire space $\mathbb{C}^2$ or it is a one-dimensional hypersurface. Based on this fact, the authors of \cite{dezeeuw} observed that a  polynomial $p(x) \in \mathbb{C}[x_1,x_2,x_3,x_4]$ is $(2,2)$-reducible if and only if there exist polynomials
$g_1 \in \mathbb{C}[x_1,x_2]$ and $g_2 \in \mathbb{C}[x_3,x_4]$ of degree at
least one, and $h_1,h_2 \in \mathbb{C}[x_1,x_2,x_3,x_4]$ such that it holds
\[
  p(x)= g_1(x_1,x_2)h_1(x) + g_2(x_3,x_4) h_2(x) . 
\]
The authors of \cite{dezeeuw} raise the following question: Is there
an algorithm that decides if a $p(x) \in \mathbb{C}[x_1,x_2,x_3,x_4]$
is $(2,2)$-reducible?  The algorithm in Sec.~\ref{sec:algo} answers this question in a more general setting and proves the following theorem.
\begin{theorem} [Symbolic Algorithm] \label{symbolicalgorithm}
  Let $\lambda=(\lambda_1,\lambda_2,\ldots,\lambda_m)$ be an $m$-partition of $n$ and let  
  $p \in \mathbb{C}[{x}_1,\ldots, {x}_n]$ be a polynomial of degree $d$.  There exists a symbolic algorithm
  \textsc{is\_decomposable} (Alg.~\ref{alg:is_decomposable}) which
  for a given polynomial $p$
  finds irreducible polynomials $g_i \in \mathbb{C}[\bar{x}_i]$
  with $deg g_i \geq 1$, $\bar{x}_i  = x_{\lambda_1 + \cdots + \lambda_{i-1} +1}, \dots,
  x_{\lambda_1 + \cdots + \lambda_{i-1} +\lambda_i}$
  , $x = x_1, \dots, x_n$,
  such that
  \[ p(x) = \sum_{i=1}^m g_i(\bar{x}_i) \, h_i(x) ,\]
  or certifies that no such decomposition exists. 
\end{theorem}

Theorem \ref{symbolicalgorithm} solves the following problem:  Given $p \in \mathbb{C}[x_1,x_2,\ldots,x_n]$, are there irreducible hypersurfaces $V_i \subset \mathbb{C}^{\lambda_i}$ such that $V_1 \times V_2 \times \ldots \times V_m \subset Z(p)$?  This covers all $\lambda$-reducible polynomials only if $\lambda=(2,2,\ldots,2)$.  For the moment, we leave it as on open problem to develop an algorithm that detects all $\lambda$-reducible polynomials for arbitrary $\lambda$.

The algorithm \textsc{is\_decomposable}
(Alg.~\ref{alg:is_decomposable}) uses some standard tools from
computer algebra, namely multivariate resultants, Canny's generalized
characteristic polynomial \cite{Canny-GCP-90}, and regular
sequences. We introduce these algebraic tools in the preliminaries
(Section~\ref{sec:-alg-prelim}).  We are now ready to state our main
result.
\begin{theorem}[Multivariate Schwartz-Zippel Lemma]
  \label{thm-SZ}
  Let $\lambda=(\lambda_1,\lambda_2,\ldots,\lambda_m)$ be an
  $m$-partition of $n$, let $S_i \subseteq \mathbb{C}^{\lambda_i}$ be
  finite sets, and let $S:=S_1 \times S_2 \times \ldots \times S_m$ be
  the multi-grid defined by $S_i$. Then, for a $\lambda$-irreducible
  polynomial $p$ of degree $d \geq 2$ and for every $\varepsilon >0$
  we have
  \[
    \abs{Z(p) \cap S} = O_{n,d,\varepsilon} \left( \prod_{i=1}^m
      \abs{S_i}^{1-\frac{1}{\lambda_i+1}+\varepsilon} + 
      \sum_{i=1}^m \prod_{j \neq i} \abs{S_j} \right),
  \]
  where the $O_{n,d,\varepsilon}$  hides constants depending on $\varepsilon,d$ and $n$.
\end{theorem}

We dedicate Section~\ref{sec:proof-mSZ} to the proof of this result.

\subsection{Applications in Combinatorial Geometry}
\label{sec:appl-CG}

A gem in incidence geometry is the Szemer\'edi-Trotter theorem on
the number of incidences between points and lines in the real plane
\cite{szemeredi-trotter}. We recover this theorem in the complex plane
except for an $\varepsilon$ in the exponent (this complex version
seems to be first proved by T{\'o}th \cite{toth} without the $\varepsilon$'s).
\begin{corollary}[Complex Szemer\'edi-Trotter Theorem]
  \label{cor:ST-in-C}
  Let $P$ be a set of points and  $L$  a set of lines in the complex plane $\mathbb{C}^2$.
  If $\mathcal{I}(P,L)$ is the set of point-line incidences, then it holds
  \[
    \abs{\mathcal{I}(P,L)} = O_{\varepsilon}\left(
      \abs{P}^{\frac{2}{3}+\varepsilon}\abs{L}^{\frac{2}{3}+\varepsilon}
      + \abs{P} + \abs{L} \right) .
  \]
\end{corollary}
\begin{proof}
  Let $p(x_1,x_2,x_3,x_4)=x_1+x_2x_3+x_4$. It is easy to prove $p$ is  $(2,2)$-irreducible. Observe that for a given point $z=(z_1,z_2)\in\mathbb{C}^2$ and a line $x+ay+b=0$ with non-zero slope, we have an incidence between $z$ and the line if and only if $p(z_1,z_2,a,b)=0$. Theorem \ref{thm-SZ} gives the bound above for point-line incidences for a set of lines with non-zero slopes. There are at most $\abs{P}$ incidences between points and zero-slope lines, so the claimed bound holds for any set of lines $\abs{L}$.
\end{proof}
Our next application is inspired by the unit distances theorem of
Spencer, Szemer\'edi and Trotter \cite{spencer}. This classical result
shows that for a given set of points $P$ in the real plane, the number
of pairs in $P \times P$ that has a unit distance admits a
sub-quadratic upper bound. The next corollary shows that a similar
statement holds for $(n,n)$-irreducible polynomials.
\begin{corollary}
  Let $P$ be a finite set of points in $\mathbb{C}^n$ and
  consider a degree $d$ polynomial  
  $f \in \mathbb{C}[x_1,\ldots,x_n,y_1,\ldots,y_n]$,
  such that $q(x,y)=f(x,y)-1$ is $(n,n)$-irreducible. Then, we have
  \[
    \abs{ \{ (u,v) \in P \times P : f(u,v) =1 \} } = O_{n,d,\varepsilon}
    \left(  \abs{P}^{2-\frac{2}{n+1}+\varepsilon}  +  \abs{P} \right) ,
  \]
  where $\varepsilon >0$ is arbitrary, and the constant hidden in
  $O_{n,d,\varepsilon}$ depends on $\varepsilon$, $d$, and $n$.
\end{corollary}
One can also consider the same result for any $\lambda$-irreducible
polynomial.
\begin{corollary}[Repeated Values of Polynomials]
  Let $\lambda=(\lambda_1,\lambda_2,\ldots,\lambda_m)$ be an
  $m$-partition of $n$, let $S_i \subseteq \mathbb{C}^{\lambda_i}$ be
  finite sets, and let $S:=S_1 \times S_2 \times \ldots \times S_m$ be
  the multi-grid defined by $S_i$. Let $p$ be a polynomial of degree
  $d$, assume that $q=p-1$ is a $\lambda$-irreducible
  polynomial. Then, for every $\varepsilon > 0$ we have
  \[ \abs{ \{ x \in S : p(x) =1 \} } = O_{n,d,\varepsilon}\left( \prod_{i=1}^m \abs{S_i}^{1-\frac{1}{\lambda_i+1}+\varepsilon} +  \sum_{i=1}^m  \prod_{j \neq i} \abs{S_j} \right)   .\] 
\end{corollary}
The following result is not a direct corollary of our main theorem,
but it can be proved by a minimal adaptation of our proof. We include
the result here for purely aesthetic reasons.
\begin{proposition}[Sparse Hypersurface-Point Incidence Theorem]
  \label{cute}
  Let $A= \{ a_1 , a_2, \ldots, a_k \}$ be a set of lattice points in
  $\mathbb{Z}_{\geq 0}^n$ with $\sum_{j=1}^n a_{ij} \leq d$ for all
  $1 \leq i \leq k$. We say a polynomial $f$ is supported in $A$ if
  $ f(x)=\sum_{i=1}^k c_i x^{a_i} $ where $c_i \in \mathbb{C}$ and
  $x^{a_i}:=x_1^{a_{i1}} x_1^{a_{i2}}\ldots x^{a_{in}}$. Let $P$ be a
  set of points in $\mathbb{C}^n$, $L$ be a set of polynomials
  supported with $A$, and let $\mathcal{I}(P,L)$ denote the collection
  of incidences between $P$ and $L$.  We assume for any sets
  $U_1 \subset P$ and $U_2 \subset L$ with $\abs{U_1}>d^n$ and
  $\abs{U_2}>d^k$, $U_1 \times U_2 $ is not included in
  $\mathcal{I}(P,L)$.  Then,
  \[ \abs{\mathcal{I}(P,L)} = O_{n,k,d,\varepsilon}\left( \abs{P}^{1-\frac{1}{n+1}+\varepsilon} \abs{L}^{1-\frac{1}{k+1}+\varepsilon} + \abs{P} + \abs{L} \right) . \]
\end{proposition}
\begin{proof}[Proof sketch]
  We define $p(x,y):=\sum_{i=1}^k y_ix^{a_i}$ where
  $y=(y_1,y_2,\ldots,y_k)$ represents polynomials in $L$, and
  $x=(x_1,x_2,\ldots,x_n)$ represents points in $P$. The polynomial
  $p(x,y)$ is not $(n,k)$-irreducible and our main theorem does not
  directly apply. However, the assumption that large Cartesian
  products are not included in $I(P,L)$ forces $p(x,y)$ to behave as
  $(n,k)$-irreducible on the set $P \times L$ (see
  Lemma~\ref{witness}). One can then prove Proposition~\ref{cute} by
  repeating the steps in the proof of Theorem~\ref{thm-SZ} with
  minimal adaptations using the assumption that large Cartesian
  products are not included instead of $(n,k)$-irreducibility.
\end{proof}

\section{Algebraic preliminaries}
\label{sec:-alg-prelim}

In this section we present some tools from computational algebraic geometry, real algebraic geometry, and commutative algebra that will be later used in the symbolic algorithm and in the  proof of Theorem~\ref{thm-SZ}.
\subsection{Resultants and Generalized Characteristic Polynomial}
\label{sec:prelim-resultant}

For a polynomial $f$, respectively a polynomial system $(F)$, we denote by
$Z(f) \subseteq \CC^n$, respectively $Z(F) \subseteq \CC^n$, its zero set.
Suppose a polynomial system $(F)$ consisting of $m$ equations $f_i = 0$ in $n$
variables, where $m \leq n$.
Every non-empty component of $Z(F)$ has dimension at least $n - m$
\cite[Cor.~3.14]{Mumford-AG-I-95}. The \textit{proper} components are the ones
with dimension exactly $n - m$. We call the components of dimension greater than
$n -m$ \textit{excess} components.
From the system $(F)$ we can eliminate $m - 1$ of the $n$ variables and obtain a single
polynomial in the remaining $n -m +1$ variables. This polynomial vanishes if the polynomials
in $(F)$ have a common zero. We call this polynomial \textit{resultant}.
We can think of the resultant as a projection operator.
It projects the algebraic set defined by the input polynomial system
that lives in dimension $n$
to an algebraic set that lives in dimension $n - m + 1$, by eliminating $m+1$ variables.

The resultant is one of the most important tools in (computational) algebraic
geometry and it is well-defined for a square system of homogeneous polynomials. We can
compute it efficiently as the determinant or a non-trivial divisor of the determinant of a matrix, see \cite{GKZ,EmiMou-met-99} and it provides a necessary and sufficient condition for the existence of a solution in the projective space. Things change considerably when one seeks solutions over the affine space. If we homogenize the polynomials to obtain a square homogeneous system and compute the resultant, then the resultant might vanish even when there are no affine solutions. The reason for this is the
presence of solutions at infinity in the projective closure of the affine zero set
or the presence of excess components at infinity.

If the number of the (affine) polynomials is less than or equal to the number of
variables, $m \leq n$, then we can consider only the $m-1$ variables and regard the
rest $n - m + 1$ as parameters. Then, we obtain an affine system of $m$ equations
in $m-1$ variables and we can use usual techniques \cite{GKZ,EmiMou-met-99} to compute the resultant;
which would be a polynomial in $n -m +1$ variables.
However, if there are (affine or projective) excess components, i.e., components of dimension greater
than $n -m$, then the resultant vanishes identically; that is, it is always zero
independently of whether there is a component of dimension $n -m$ or not and
so it gives no information.

Canny \cite{Canny-GCP-90} introduced a projection operator that he called \textit{Generalized
Characteristic Polynomial (GCP)} to overcome the issues of the resultants, in particular to avoid the identical vanishing in the presence of excess components. The idea is to perturb symbolically the
polynomials of the initial system using a new parameter, say $s$, then compute the resultant of the perturbed system and consider the resultant as a polynomial in $s$.
The perturbation is such that it guarantees the new resultant polynomial is not identically
zero. The coefficient of the lowest degree term in $s$ is the projection operator of interest.
If there are no excess components, then  the constant coefficient (with respect to $s$)
is the projection operator and coincides with the classical resultant.

The GCP guarantees that we recover all the proper components of the
intersection, that is the components of the expected dimension; we refer the
reader to \cite[Thm.~3.2]{Canny-GCP-90} for the proof and further details.
However, it might also contain additional proper components, that live in the
excess components; these components can also be projected in the resultant
polynomial, but they do not affect our algorithms. Nevertheless, we can identify
(in a randomized way) the additional components by performing many random
perturbations.
In what follows, for a polynomial system $f_1, \dots, f_m$ with
polynomials in $\CC[x_1, \dots, x_n]$, of degrees $d_1, \dots, d_m$,  the operator
$\Elim(f_1, \dots, f_m : x_1, \dots, x_m)$
realizes the technique of GCP. It eliminates the variables  $x_1, \dots, x_m$ and results a non-identically zero
polynomial $R \in \CC[x_{m+1}, \dots, x_n]$.
It proceeds as follows:
First, it perturbs  symbolically the polynomials,
that is it constructs the polynomials $\hat f_i = s x_i^{d_i} + f_i$ for in $i \in [m]$.
Then it computes the  resultant of $\hat f_1, \dots, \hat f_m$
that eliminates the variables $x_1, \dots, x_m$;
the resultant is $\hat R \in (\CC[x_{m+1}, \dots, x_n])[s]$.
Finally, it returns $R = \mathtt{tcoeff}(\hat R, s) \in \CC[x_{m+1}, \dots, x_n]$,
where $\mathtt{tcoeff}$ returns the trailing coefficient of a polynomial with respect to $s$.

\subsection{Some Basic Algebraic Geometry}
\label{sec:AG}



We begin with an affine version of B\'ezout's
theorem. This result is certainly well-known but we do not know who
was the first person to write it down.  We state the version of the
result we need below. For square systems (the case $m=n$), an affine
B\'ezout's inequality was subject to a well written paper by Schmidt
\cite{schmidt}. The more general version we need can also be proved by
essentially repeating Schmidt's argument, where $u$-resultant in his
proof needs to be replaced by the generalized characteristic
polynomial of Canny. One can also give a proof based on regular
sequences and Schmidt's result for square systems. We will give a
sketch of the latter idea. 
\begin{theorem}[Affine B\'ezout Inequality for Overdetermined
  Systems] \label{bezout}
  Consider the polynomials $p_1,p_2,\ldots,p_m$ in 
  $n$-variables and degree at most $d$. Denote by  $V(p_1,p_2,\ldots,p_m)$ 
  the affine variety defined by $p_i$ in $\mathbb{C}^n$. Further
  assume that
\[ V(p_1,p_2,\ldots,p_m)=V_0 \cup V_1 ,\]
where $V_0$ is a pure zero-dimensional variety, and $V_1$ is either empty set or a positive dimensional variety.  
Then, we have $\abs{V_0} \leq d^n$.
\end{theorem}

\begin{proof}[Proof Sketch]
  Let $I:=\left( I(V(p_1,p_2,\ldots,p_m)):I(V_1) \right)$ be the ideal defined by the saturating the radical ideal $I(V)$ with $I(V_1)$. By definition $I=I(V_0)$, and also
  $I= \langle \tilde{f}_1, \tilde{f}_2,\ldots, \tilde{f}_m \rangle$ for some $\tilde{f}_i$ of degree at most $d$. Since $V_0$ is zero dimensional, $I= \langle \tilde{f_1},\tilde{f_2},\ldots,\tilde{f_m} \rangle$ has depth $n$; thus we can find a regular sequence 
$ \tilde{f}_{\sigma(1)}, \tilde{f}_{\sigma(2)}, \ldots , \tilde{f}_{\sigma(n)} $ that generates the ideal $I$, and are given by linear combinations of $\tilde{f}_i$. So, $V_0$ is included  in $V(\tilde{f}_{\sigma_1},\tilde{f}_{\sigma_2},\ldots,\tilde{f}_{\sigma_{n}})$, and the cardinality of this set is bounded above by $d^n$ due to the affine B\'ezout inequality.
\end{proof}

Similar results to \cref{bezout} can be found in \cite{VogPat-rit-83}. For a nice and elementary exposition on \cref{bezout} we refer the reader to Tao's blog post \cite{Tao-bezout}.
Thm.~\ref{bezout} leads to a simple criterion that needs to be satisfied by all $\lambda$-irreducible polynomials.
\begin{lemma}
\label{finite-witness}
Let $C\subset\CC^n$ be a Zariski closed subset and $d\geq 1$ be an integer. There exists a finite subset $U\subset C$ such that any degree $d$ polynomial that vanishes on $U$ also vanishes on $C$.	
\end{lemma}
\begin{proof}
Set $\mathcal{P}_d\coloneqq\{p\in\CC[x_1,x_2,\dots,x_n]\mid\deg(p)\leq d\}$ and for a subset $S\subset\CC^n$, define $I_d(S)\coloneqq\{p\in\mathcal{P}_d\mid p(s)=0, \forall x\in S\}$. Let $U\subset C$ be a finite subset with minimal $\dim I_d(U)$. Then, for each $x\in C$ we have $I_d(U)=I_d(U\cup\{x\})$. In particular, if $p\in I_d(U)$ then $p\in I_d(U\cup\{x\})$ so $p(x)=0$. As $x\in C$ is arbitrary, we deduce that $p$ vanishes on $C$; thus, $I_d(U)=I_d(C)$.
\end{proof}

Thm.~\ref{bezout} leads to a simple criterion that needs to be satisfied by all $\lambda$-irreducible polynomials.
\begin{lemma}
  \label{witness}
  Let $\lambda=(\lambda_1, \lambda_2, \ldots, \lambda_m)$ be a partition of $n$,
  i.e., $n=\lambda_1 + \ldots + \lambda_m$. Let
  $U_i \subset \mathbb{C}^{\lambda_i}$ be finite sets with
  $\abs{U_i} > d^{\lambda_i}$ for some integer $d \geq 1$. Let $p$ be a polynomial of degree
  $d$ in $n$ variables. If $p$ vanishes entirely on the set
  \[ U_1 \times U_2 \times \ldots \times U_{m} \]
  then $p$ is $\lambda$-reducible.
\end{lemma}
\begin{proof}
  For every point $y \in \mathbb{C}^{n-\lambda_1}$, we define a polynomial $p_y$
  in $\lambda_1$ variables by setting $p_y(x)=p(x,y)$. We know that
  \[ \abs{\cap_{y \in U_2 \times U_3 \times \ldots \times U_m}   V(p_y) } \geq \abs{U_1} > d^{\lambda_1}   .\]

  By \cref{bezout}, there exists a positive dimensional variety $C_1 \subset \mathbb{C}^{\lambda_1}$ such that
  \[ C_1 \times U_2 \times U_3 \times \ldots \times U_m \subset V(p) . \]
  
  We use \cref{finite-witness} to obtain a finite subset $\tilde{U}_1\subset C_1$ with the property that any degree $d$ polynomial that vanishes on $\tilde{U}_1$  also vanishes on $C_1$. 

  Now for
  every point $y \in \tilde{U}_1 \times U_3 \times U_4 \times \ldots \times U_m$, we
  consider a polynomial $p_y$ in $\lambda_2$ variables by setting
  $p_y(x)=p(x,y)$, and we repeat the argument above using \cref{bezout}. This
  shows the existence of a positive dimensional variety $C_2$ with
  \[ \tilde{U}_1 \times C_2 \times U_3 \times \ldots \times U_m \subset V(p) .\] On the other hand, if $y\in C_2\times U_3\times\ldots\times U_m$, then $p_y$ vanishes on $\tilde{U}_1$ and has degree at most $d$. By the defining property of $\tilde{U}_1$, we have $C_1\subset V(p_y)$, and hence \[
  C_1\times C_2\times U_3\times \ldots\times U_m\subset V(p).
  \]
  
  We can iterate this argument: At the $i$th step, we pick a finite subset $\tilde{U}_i\subset C_1\times C_2\times \ldots\times C_{i-1}$ with the desired property using \cref{finite-witness}, we then observe that for any point $y\in \tilde{U}_i\times U_{i+1}\times\ldots \times U_{m}$, $p_y$ vanishes on $U_{i}$. Then, using \cref{bezout}, we deduce that $p_y$ vanishes on a positive dimensional variety $C_i$. But now, for each $y\in C_i\times U_{i+1}\times\ldots\times U_m$, $p_y$ vanishes on $\tilde{U}_i$; therefore, \[C_1\times C_2\times \ldots\times C_i\times U_{i+1}\times \ldots\times U_m\subset V(p). \] After the $m$th step, we have constructed positive dimensional varieties $C_1,C_2,\dots,C_m$ such that $C_1\times C_2\times\ldots\times C_m\subset V(p)$ and obtained the desired result. 
\end{proof}

\subsection{Tools from Real Algebraic Geometry and Polynomial Partitioning}
We first present a useful tool invented by Guth and Katz in their
solution to Erd\H{o}s distinct distances problem \cite{guth-katz}, see \cite{tao-solymosi} for a nice exposition (in particular Corollary 5.3).
\begin{proposition}[Polynomial Partitioning Lemma] \label{GK}
Let $Q \subset \mathbb{R}^n$ be a finite set of points, and let $d \geq 2$ be an integer. Then, there exist a polynomial $p \in \mathbb{R}[x_1,x_2,\ldots,x_n]$ with degree at most $d$ and a partition   
\[ \mathbb{R}^n =  Z(p) \cup \Omega_1 \cup \Omega_2 \cup \ldots \cup \Omega_M ,\]
such that the boundary of each set $\Omega_i$ is in $Z(p)$, and $\abs{Q \cap \Omega_i} \leq \abs{Q} / d^n$, for all $i=1,2,\ldots,M$. 
\end{proposition}

We will need a refinement of polynomial partitioning lemma due to Fox, Pach, Sheffer, Suk, and Zahl.
\begin{lemma}[Theorem 4.2 \cite{fox}] \label{zahl}
Let $V \subset \mathbb{R}^n$ be an irreducible algebraic variety of dimension $k$ and assume the degree of the complexification of $V$ has is $\delta$. Let $Q$ be a collection of points in $\mathbb{R}^n$. Then there exists a polynomial $p$ that does not vanish entirely on $V$ such that $deg p \leq C(k,\delta) d$ for a constant $C(k,\delta)$ that depends only on $k,\delta$ and 
\[ \mathbb{R}^n =  Z(p) \cup \Omega_1 \cup \Omega_2 \cup \ldots \cup \Omega_M ,\]
such that the boundary of each set $\Omega_i$ is in $Z(p)$, and $\abs{Q \cap \Omega_i} \leq \abs{Q} / d^n$, for all $i=1,2,\ldots,M$. 
\end{lemma}

Now we recall a result from real algebraic geometry which has been proved in a more general form by several authors (see for instance Appendix Theorem A.2 of \cite{tao-solymosi} and \cite{barone-basu}).

\begin{theorem} \label{SS}
Let $p \in \mathbb{R}[x_1,\ldots,x_n]$ be a polynomial of degree $d$, and let $V \subseteq \mathbb{R}^n$ be the real part of a $k$-dimensional complex variety. Then the semi-algebraic set $\{ x \in V : p(x) \neq 0 \}$ has at most $O_{n,k}(d^{k})$ many connected components.
\end{theorem}

\section{Symbolic Algorithm}
\label{sec:algo}

To simplify the presentation we modify the notation for the variables.
Let $\x_i = (x_{i,1}, \dots, x_{i,n})$ for $i \in [m+1]$. That is we assume
$\lambda=(n,n,\ldots,n)$ is an $m+1$-partition of $n(m+1)$.  
In this section we consider polynomials with rational coefficients
because our algorithms rely on the irreducibility of polynomials over the rationals.  If we allow complex coefficients, then we have to rely on polynomials that are irreducible over the complex numbers, which is a stronger condition (known as absolute irreducibility).

To summarize, we are interested in an algorithm that solves the
following problem.
\begin{problem} \label{algorithm}
  
  Consider a polynomial $F \in \QQ[\x_1, \dots, \x_{m+1}]$ of degree
  $d$.  Are there polynomials $G_i \in \QQ[\x_i]$ of degree $\delta \geq 1$
  irreducible over the rationals and $H_i \in \QQ[\x_1, \dots, \x_{m+1}]$ with $\deg( H_i) \leq d$, where $i \in [m+1]$,
  such that we can write $F$ as
  \begin{equation}
    F(\x_1, \dots, \x_{m+1}) = \sum_{i=1}^{m+1} G_i(\x_i) \, H_i(\x_1, \dots, \x_{m+1}) .
    \label{eq:F-decomp}
  \end{equation}
  Equivalently, are there hypersurfaces $\cV_i \subset \CC^n$ such that
  \[
    \cV_1 \times \cV_2 \times \cdots \times \cV_{m+1}
    \subset \VV(F) \subset \CC^{(m+1)n},
  \]
  where $\cV_i = \VV(G_i) \subset \CC^{n}$ are the zero sets of the polynomials $G_i$, $i \in [m+1]$.
\end{problem}

The equivalence of the algebraic and geometric formulation in Problem \ref{algorithm} follows by the following proposition. The proof is similar to the proof of \cite[Theorem 2.2]{dezeeuw} in which the authors prove the case $n=4,m=2$. 
\begin{proposition}
\label{prop:H-dg-bound}
  Let $G_i\in\CC[\x_i], i=1,2,\dots,m$ be irreducible polynomials each having degree at least $1$. Then if \[
  Z(G_1)\times Z(G_2)\times\dots\times Z(G_m) \subseteq Z(F) ,
  \] then there exist polynomials $H_1,H_2,\dots, H_m$ of degree at most $\deg(F)$, such that \[
  F = G_1 H_1+G_2 H_2+\dots+G_m H_m.
  \]
\end{proposition}
\begin{proof}
We proceed by induction on $m$. For the base case $m=1$, we note that $Z(F)\supset Z(G_1)$ implies $F\in \CC[\x_1]\cdot G_1$ and hence $F=G_1 H_1$ since $G_1$ is irreducible. The degree bound follows by $\deg F=\deg G_1+\deg H_1$. 

For $m>1$, we fix a monomial ordering that prioritizes the variables $\x_i$ over $\x_j$, for $i<j$, and respects the degree of monomials; that is, if $x^\alpha<x^\beta$ then $|\alpha|\leq |\beta|$. Using polynomial divison, we write \[
F = G_1 H_1+R
\] where no monomial of $R$ is divisible by the leading term of $G_1$. Note that since the monomial ordering respects the degree, we have $\deg R\leq \deg F$ and $\deg H_1\leq \deg F$.

For a point $p$ such that  
\[ p\in Z(G_2)\times\dots\times Z(G_m), \] 
$F(\x_1,p)$ vanishes on $Z(G_1)$, by the assumption on $F$. Hence, 
\[ R(\x_1,p) = F(\x_1,p)-G_1(\x_1)H_1(\x_1,p) \]
which implies that $R(\x_1,p)$ also vanishes on $Z(G_1)$ and it is divisible by $G_1$. On the other hand, each monomial of $R(\x_1,p)$ divides a monomial of $R$ and by the assumption that no monomial of $R$ is divisible by $LT(G_1)$, this is only possible when $R(\x_1,p)=0$. Write 
\[ R = \sum_{\alpha} R_\alpha \x_1^\alpha, \] so we have
\[   Z(G_2)\times\dots\times Z(G_m) \subseteq Z(R_\alpha). \] 
By the induction hypothesis, $R_{\alpha}$ can be written as follows 
\[ R_\alpha = G_2 H_{\alpha,2}+\dots+G_m H_{\alpha,m} \] 
where $\deg H_{\alpha,i}\leq \deg R_\alpha\leq \deg R-|\alpha|\leq \deg F-|\alpha|$. Then we have \[
F = G_1 H_1 + G_2 (\sum_{\alpha} H_{\alpha,2} \x_1^\alpha)+\dots+G_m (\sum_{\alpha} H_{\alpha,m}\x_1^\alpha).
\] Setting $H_i=\sum_\alpha H_{\alpha,i}\x_1^\alpha$, the result follows.
\end{proof}

\begin{algorithm2e}[t]
  \SetFuncSty{textsc}
  \SetKw{RET}{{\sc return}}
 \KwIn{
    $F(\x_1, \dots, \x_m, \x_{m+1})$
  }
  
  \KwOut{
    \texttt{TRUE}
    if we can write $F$ as 
    $F = \sum_{i=1}^{m+1} G_i(\x_i) \, H_i(\x_1, \dots, \x_m, \x_{m+1})$, 
    \texttt{FALSE} otherwise. }
    
  \BlankLine
  
  $\cG_1, \dots, \cG_m \leftarrow
  \FuncSty{recover\_m}(F)$ \;

  $\cG_{m+1} \leftarrow
  \FuncSty{recover\_last}(F)$ \;
  
  \tcc{{\small Perform a square-free decomposition and factorization to each $\cG_k$.}}
  \For{$1 \leq k \leq m+1$}{
      $\cG_k  = G_{k,1}^{\delta_{k_1}} \cdots G_{x,s_k}^{\delta_{s_k}} $ \;
  }

  \BlankLine
  
  \For{$i_1 \in [s_1], \dots, i_m \in [s_m], i_{m+1} \in [s_{m+1}]$}{
     Check (using linear algebra) if there are $H_1, \dots, H_{m+1}$ (of degree $\leq d$)
    such that
    \[
      F = G_{i_1}\, H_1 + \dots + G_{i_m} \, H_m + G_{i_{m+1}} \, H_{m+1}
    \]
    If this is the case, then \RET \texttt{TRUE}
  }
  \RET \texttt{FALSE}  
  \caption{\FuncSty{is\_decomposable}}
  \label{alg:is_decomposable}
\end{algorithm2e}

The algorithm \textsc{is\_decomposable} (Alg.~\ref{alg:is_decomposable}) provides a solution to the problem \ref{algorithm}.
It depends on two sub-algorithms. The first one \textsc{recover\_m} (Alg.~\ref{alg:recover_m}) recovers candidates for $G_1, \dots, G_m$
and the second one, \textsc{recover\_last} (Alg.~\ref{alg:recover_last}) provides candidates for $G_{m+1}$.
Given the various candidates for $G_1, \dots, G_m$, the algorithm
\textsc{is\_decomposable} checks if there is a certain combination of
them that allows us to write $F$ as in \eqref{eq:F-decomp}.


\begin{algorithm2e}[ht]
  \SetFuncSty{textsc}
  \SetKw{RET}{{\sc return}}
  \KwIn{
    $F(\x_1, \dots, \x_m, \x_{m+1})$
  }
  
  \KwOut{
    $\cG_{\nu} \in \CC[\x_{\nu}]$ with $\nu \in m$
    such that if
    we can write $F$ as 
    $F = \sum_{i=1}^{m+1} G_i(\x_i) \, H_i(\x_1, \dots, \x_m, \x_{m+1})$,
    then $G_{\nu}(\x_{\nu})$ is a factor of $\cG_{\nu}(\x_{\nu})$. }
    
  \BlankLine

  \tcc{\small Introduce new variables
    $\x_1^{(1)}, \dots, \x_m^{(1)}, \dots \x_1^{(mn)}, \dots, \x_m^{(mn)}$}

  $\MM \leftarrow 
  \left[
    \begin{array}{llllll}
      \x_1^{(1)} & \x_2^{(1)} & \cdots & \x_m^{(1)} \\
      \vdots  & \vdots & & \vdots \\
      \x_1^{(mn)} & \x_2^{(mn)} & \cdots & \x_m^{(mn)} \\
    \end{array}
  \right] 
  $ 
  
  $M \leftarrow 1 - t \, \det(\MM) $ \;
  
  \BlankLine

  \For{$1 \leq j \leq mn$} {
    $F_j \leftarrow F( \x_1^{(j)}, \dots, \x_m^{(j)}, \x_{m+1})$ ;
  }

  \For{$1 \leq \nu \leq m$ }{
    \For{$1 \leq k \leq (m n)^2 - m n +1$} {
      $L_k(\x_{m+1}) \leftarrow \ell_{k,0} + \ell_{k,1}x_{m+1, 1} + \cdots + \ell_{k,n}x_{m+1, n}$ ;
     %
      %
      $R_k \leftarrow
      \Elim(F_1, \dots, F_{mn}, M, L_k \,:\, \{\x_{j}^{(1)} \}_{1 \leq j \leq m, j\not=\nu}, \x_{m+1}, t)$ \;
    }
    $\cG_{\nu} \leftarrow \Elim(R_1, \dots, R_{(mn)^2 - mn + 1} \,:\,
    \x_1^{(2)}, \dots, \x_{m}^{(2)}, \dots, \x_{1}^{(m n)}, \dots, \x_{m}^{(m n)})
    \in \QQ[\x_{\nu}^{(1)}]$ 
    Rename the variables so that $\cG_{\nu} \in \QQ[\x_{\nu}]$ \;
  }
      
  \RET $\cG_1, \dots, \cG_m$
  
  \caption{\FuncSty{recover\_m}}
  \label{alg:recover_m}
\end{algorithm2e}

\begin{algorithm2e}[t]
  \SetFuncSty{textsc}
  \SetKw{RET}{{\sc return}}
  \KwIn{
    $F(\x_1, \dots, \x_m, \x_{m+1})$
  }
  
  \KwOut{
    $\cG_{m+1} \in \CC[\x_{m+1}]$ 
    such that if
    we can write $F$ as 
    $F = \sum_{i=1}^{m+1} G_i(\x_i) \, H_i(\x_1, \dots, \x_m, \x_{m+1})$,
    then $G_{m+1}(\x_{m+1})$ is a factor of $\cG_{m+1}(\x_{m+1})$. }
    
  \BlankLine

  \tcc{\small Introduce new variables
    $\x_2^{(1)}, \dots, \x_{m+1}^{(1)}, \dots, \x_2^{(mn)}, \dots, \x_{m+1}^{(mn)}$}

  $\MM \leftarrow 
  \left[
    \begin{array}{llllll}
      \x_2^{(1)} & \x_3^{(1)} & \cdots & \x_{m+1}^{(1)} \\
      \vdots  & \vdots & & \vdots \\
      \x_2^{(mn)} & \x_3^{(mn)} & \cdots & \x_{m+1}^{(mn)} \\
    \end{array}
  \right] 
  $ 
  
  $M \leftarrow 1 - t \, \det(\MM) $ \;
  
  \BlankLine

  \For{$1 \leq j \leq mn$} {
    $F_j \leftarrow F( \x_1, \x_2^{(j)}, \dots, \x_{m+1}^{(j)})$ ;
  }

  \For{$1 \leq k \leq (m n)^2 - m n +1$} {
    $L_k(\x_{1}) \leftarrow \ell_{k,0} + \ell_{k,1}x_{1, 1} + \cdots + \ell_{k,n}x_{1, n}\ ;$ 
    %
    %
    $R_k \leftarrow
    \Elim(F_1, \dots, F_{mn}, M, L_k \,:\, \x_1, \x_2^{(1)}, \dots, \x_m^{(1)}, t)$ \;
  }
  
  $\cG_{m+1} \leftarrow \Elim(R_1, \dots, R_{(mn)^2 - mn + 1} \,:\,
  \x_2^{(2)}, \dots, x_{m+1}^{(2)}, \dots, \x_2^{(m n)}, \dots, \x_{m+1}^{(m n)}) \in \QQ[\x_{m+1}^{(1)}]$ \;

    Rename the variables so that $\cG_{m+1} \in \QQ[\x_{m+1}]$ \;

  \RET $\cG_{m+1}$
  
  \caption{\FuncSty{recover\_last}}
  \label{alg:recover_last}
\end{algorithm2e}

Let $(\bm{\alpha}_1^{(j)}, \dots, \bm{\alpha}_m^{(j)})  \in 
\cV_1 \times \cV_2 \times \cdots \cV_m  \subset (\CC^n)^{m}$
be an $m$-tuple of (complex) points.
Also let $\mathcal{L} = \{(\bm{\alpha}_1^{(j)}, \dots, \bm{\alpha}_m^{(j)}) \}_{j \in [N]}$
be a set containing $N$ such $m$-tuples, for some positive $N$.
Now consider the restriction of $F$ at these $N$ tuples; 
this results $N$ polynomials in $\x_{m+1}$,
that is
$F_1= F(\bm{\alpha}_1^{(1)}, \dots,  \bm{\alpha}_m^{(1)}, \x_{m+1}),
\dots, F_N = F(\bm{\alpha}_1^{(N)}, \dots, \bm{\alpha}_m^{(N)}, \x_{m+1})$.
and the system  $(\Sigma)$ consisting of the  polynomials
\[
  (\Sigma) : \, \{ F(\bm{\alpha}_1^{(1)}, \dots, \bm{\alpha}_m^{(1)}, \x_{m+1}),
  \dots, F(\bm{\alpha}_1^{(N)}, \dots, \bm{\alpha}_m^{(N)}, \x_{m+1}),
  L(\x_{m+1})  \},
\]
where
$L(\x_{m+1}) = \ell_{0} + \ell_1x_{m+1, 1} + \cdots + \ell_nx_{m+1,n}$
is a generic linear polynomial.
By generic we mean that the coefficients of $L$
belong to an (Zariski) open subset so that $L$ does not intersect the variety defined by the polynomials $F_1, \dots, F_N$ in any special way.
We can guarantee this by picking them uniformly at random from a sufficiently large finite subset of $\ZZ$.
Thus, $(\Sigma)$ consists of
$N+1$ polynomials in $n$ variables, i.e.,
$\x_{m+1} = (x_{m+1, 1}, \dots, x_{m+1,n})$.
Since $(\bm{\alpha}_1^{(j)}, \dots, \bm{\alpha}_m^{(j)})  \in 
\cV_1 \times \cV_2 \times \cdots \cV_m  \subset (\CC^n)^{m}$, and $\cV_{m+1}$ is positive dimensional as it is a hypersurface, we conclude that  $(\Sigma)$ is not empty, that is
$\VV(\Sigma) \not= \emptyset$. Our goal is to search for suitable sets of ($m$-tuples of) points
$\{(\bm{\alpha}_1^{(j)}, \dots, \bm{\alpha}_m^{(j)}) \}$ that have this property. These sets will be the candidates for the sets $\cV_1, \dots, \cV_m$.

We present in detail how to extract candidates for the polynomial $G_1$, or
equivalently for the set $\cV_1$. The algorithm \textsc{recover\_m}
(Alg.~\ref{alg:recover_m}) does this, and in addition it computes candidates for
the polynomials $G_2, \dots, G_m$. This is the semantics of the loop over the
variable $\nu$, where at each iteration we compute a polynomial
$\mathcal{G}_{\nu}$ that contains the corresponding $G_{\nu}$ as a factor.

Let $\x_i^{(j)}$ be $n$-tuples of new variables, where $i \in [n]$ and
$j \in [mn]$. We arrange the new variables in a square $mn \times mn$
matrix $\MM$ as follows:
\begin{equation}
  \label{eq:Xvar-mat}
  \MM = 
  \left[
  \begin{array}{llllll}
    \x_1^{(1)} & \x_2^{(1)} & \cdots & \x_m^{(1)} \\
    \vdots  & \vdots & & \vdots \\
    \x_1^{(mn)} & \x_2^{(mn)} & \cdots & \x_m^{(mn)} \\
  \end{array}
  \right] .
\end{equation}

We consider the polynomial $M(\x_j^i, t) = 1 - t\, \det(\MM)$.  It
serves the following purpose: If for a given set of points
$\x_1^{(1)}, \dots, \x_m^{(1)}, \dots, \x_1^{(mn)}, \dots, \x_m^{(mn)}$
there exists a $t \in \CC$ such that
\[
  (\x_1^{(1)}, \dots, \x_m^{(1)}, \dots, \x_1^{(mn)}, \dots, \x_m^{(mn)}, t)
  \in \VV(M),
\]
then the rows of the matrix $\MM$ are linearly independent. This will
prevent us from constructing the trivial zero set by repeating twice
or more the same polynomial.

We need to introduce $(m n)^2 - m n + 1$ generic linear polynomials in
the variables $\x_{m+1}$, say $L_k(\x_{m+1})$, for
$1 \leq k \leq (m n)^2 - m n + 1$.  Consider the linear polynomial
$L_1$ and the polynomial system
\begin{equation}
  \label{eq:S1}
  (\Sigma_1) 
  \left\{
      \begin{array}{lcllll}
        F_1 =F(\x_1^{(1)}, \dots, \x_m^{(1)}, \x_{m+1})  = 0 \\
        \vdots  \\
        F_{m n} = (\x_1^{(m n)}, \dots, \x_m^{(m n)}, \x_{m+1})  = 0 \\
        M(\x_j^i, t) = 1 - t\, \det(\MM) = 0\\
        L_1(\x_{m+1}) =  0 .
  \end{array}
  \right . 
\end{equation}
The system $(\Sigma_1)$ consists of  $m n + 2$ polynomial equations in $(m n)^2 + n + 1$ variables.
Hence, using a resultant operator we can eliminate $m n + 1$ of the variables.
We choose to eliminate $\x_2^{(1)}, \dots, \x_m^{(1)}, \x_{m+1}, t$.
We denote this elimination by 
\[
  R_1 \leftarrow
  \Elim(F_1, \dots, F_{mn}, M, L_1 \,:\, \{\x_2^{(1)}, \dots, \x_m^{(1)}, \x_{m+1}, t\}) ,
\]
where $R_1 \in \QQ[\x_1^{(1)}, \x_1^{(2)}, \dots, \x_{m}^{(2)}, \dots, x_{1}^{(m n)}, \dots, \x_{m}^{(m n)}]$
has $(m n)^2 - m n + n$ variables.

We repeat this elimination process $(m n)^2 - m n +1$ times, each time
using a different linear polynomial, say $L_k(\x_{m+1})$. Thus, after
elimination, we obtain $(m n)^2 - m n +1$ polynomials $R_k$
that live in
$\QQ[\x_1^{(1)}, \x_1^{(2)}, \dots, \x_{m}^{(2)}, \dots, x_{1}^{(mn)}, \dots, \x_{m}^{(m n)}]$.

Next, we consider the polynomial system consisting of the
$(mn)^2 - m n + 1$ resultant polynomials,
$\{ R_1, \dots, R_{(mn)^2 - m n + 1} \}$, in $(m n)^2 - m n + n$
variables.  We can eliminate $(m n)^2 - m n$ of them.  We choose to
eliminate
$\x_1^{(2)}, \dots, \x_{m}^{(2)}, \dots, \x_{1}^{(m n)}, \dots,
\x_{m}^{(m n)}$ and so
\[
  \cG_1 \leftarrow \Elim(R_1, \dots, R_{(mn)^2 - mn + 1} \,:\,
  \x_1^{(2)}, \dots, \x_{m}^{(2)}, \dots, \x_{1}^{(m n)}, \dots,
  \x_{m}^{(m n)}) \in \QQ[\x_{1}^{(1)}],
\]
where $\cG_1 \in \QQ[\x_1^{(1)}]$.
We can rename the variables of $\cG_1$ and  assume that $\cG_1 \in \QQ[\x_1]$.

By construction the projection $\cV_1$ is included in $\VV(\cG_1)$.
Therefore, if we can write $F$ as in \eqref{eq:F-decomp}, then we can
recover $G_1$ as a factor of $\cG_1$, because we have assumed that it
is irreducible over the rationals.  Thus, we should perform a
factorization of $\cG_1$ over the rationals.

To recover candidates for $G_2$, when we compute the resultants $R_k$, we choose
to eliminate the variables $\x_1^{(1)}, \x_{3}^{(1)}, \dots, \x_{m}^{(1)}, \x_{m+1}, t$.
Subsequently, the elimination procedure gives us $\cG_2$, which is a polynomial in $\QQ[\x_2^{(1)}]$.
We work similarly for $\cG_3, \dots, \cG_{m}$.
To obtain $\cG_{m+1}$ we should choose a different generic linear
polynomial, but the algorithm is almost the same. We call it
\textsc{recover\_last} (Alg.~\ref{alg:recover_last}).  In this case we
introduce linear forms $L_k \in \QQ[\x_1]$ and we choose the variables
that we eliminate in order to end up with a polynomial
$\cG_{m+1} \in \QQ[\x_{m+1}]$.

Finally, \textsc{is\_decomposable} (Alg.~\ref{alg:is_decomposable})
combines the various candidates to test if it is possible to write $F$
as in \eqref{eq:F-decomp}.  This step relies purely on linear algebra.
Recall, that we know the degrees of the polynomials $H_i$.  If we are
given $G_1, \dots, G_{m+1}$, then we consider the coefficients of
$H_i$ as unknowns and we construct the linear system formed by identifying 
the coefficients of $F$ with those of the polynomial
$\sum_{i=1}^{m+1}G_i \, H_i$.  If the linear system has a solution,
then we have obtained a decomposition.

\begin{remark}
   We only need the degree bound on $H_i$ coming from Prop.~\ref{prop:H-dg-bound} because we rely on linear algebra for computing the decomposition of $F$. Alternatively, we can exploit the fact that the polynomials $G_i$ form a Gr\"obner basis and compute the $H_i$'s, following the proof of the proposition \ref{prop:H-dg-bound}, using successive polynomial divisions. In this case the degree bound is not necessary.
\end{remark}

 


\section{The proofs}
\label{sec:proofs}

\subsection{Proof of  multivariate Combinatorial Nullstellensatz}
\label{sec:proof-CN}

\begin{proof}[Proof of Theorem~\ref{combnull}]
We assume the existence of a monomial $x^{\alpha}$ with non-zero
coefficient in the expansion of $p$ with $\deg_i(x^{\alpha})=d_i$ for
all $i$. We use the notation $\bar{x}_1=(x_1,x_2,\ldots,x_{\lambda_1})$, and
similarly for $\bar{x}_i$, for $1 \leq i \leq m$. We use $\bar{x}_1^{\alpha_1}$ to
denote the monomial
$x_1^{\alpha_1} x_2^{\alpha_2} \ldots
x_{\lambda_1}^{\alpha_{\lambda_1}}$ and
$\bar{x}_2^{\alpha_2}, \ldots, \bar{x}_{m}^{\alpha_m}$ are used in the
same fashion. We will now construct some auxiliary functions for our
proof: We claim the existence of functions $f_i : S_i \mapsto \mathbb{F}$
for $i=\{1,2,\ldots,m\}$ with the following properties:
\begin{enumerate}
\item For all $\beta \in \mathbb{Z}_{\geq 0}^{\lambda_i}$ with $\deg_i(\bar{x}_i^\beta) \leq d_i$ and $\beta \neq \alpha_i$, we have
\[ \sum_{y \in S_i} f_i(y) y^{\beta} = 0 . \]
\item For the case of $\bar{x}_i^{\alpha_i}$, that is when $\beta = \alpha_i$, we have 
\[ \sum_{y \in S_i} f_i(y) y^{\alpha_i} = 1 . \]
\end{enumerate}
We prove the existence of $f_1$; the same proof works for all $f_i$.
The technique has similarities with the construction of multivariate Lagrange interpolation polynomials.
We construct a $\binom{\lambda_1 + d_1}{d_1} \times \abs{S_1}$  matrix $A$.
Each row of $A$ corresponds to a $\beta \in \mathbb{Z}_{\geq 0}^{\lambda_1}$ with $\deg_1(\bar{x}_1^{\beta}) \leq d_1$ and is of the form  
\[ ( y_1^{\beta} , y_2^{\beta}, \ldots, y_{\abs{S_1}}^{\beta}),  \]
where $y_j$ are the distinct elements of $S_1$.
The matrix $A$ is a multivariate Vandermonde matrix.
There are $\binom{\lambda_1 + d_1}{d_1}$ rows and $\abs{S_1}$ columns.
The assumption $\deg(S_1) > d_1 $ implies that
$\abs{S_1} \geq \binom{\lambda_1+d_1}{d_1}$. We will show that $A$ is full rank,
i.e., $\mathrm{rank}(A)=\binom{\lambda_1+d_1}{d_1}$. Assume that rows of $A$ are
linearly dependent. That is, we assume there exists a
$\binom{\lambda_1+d_1}{d_1} \times 1$ vector $ c = ( c_{\beta} )$ with
$c^{\top} A = 0$. Then for any $y \in S_1$ we have

\[\sum_{\deg(\beta)\leq d_1} c_{\beta} y^{\beta} = 0 .\]
If we define a polynomial $g \in \mathbb{F}[\bar{x}_1]$ by setting $g(\bar{x}_1)=\sum_{\deg(\beta)\leq d_1} c_{\beta} \bar{x}_1^{\beta}$, then $g$, which is of degree $d_1$, vanishes on the entire set $S_1$. This contradicts with the assumption that $\deg(S_1) > d_1$,
that imposes that every polynomial vanishing on $S_1$ has to have degree greater than $d_1$. Thus, the rows of $A$ are linearly independent. Now we can find the desired $f_1$ by solving the linear system 
\[ A [f_1(y_1) , f_1(y_2) , \ldots , f_1(y_{\abs{S_1}}) ]^{\top}= [0,0,\ldots,0,1]^{\top} ,\]
where the only non-zero entry in the right hand side is at the coordinate corresponding to $\alpha_1$.

Using the $f_i$'s we complete the proof as follows:
Assume that $p \in \mathbb{F}[x_1,\ldots,x_n]$ satisfies the
hypothesis of the theorem and also vanishes on the entire set $S$. Let
$p = \sum_{\beta} p_{\beta} x^{\beta}$, where the sum is over all monomials
$x^{\beta}$ with $\deg_i(x^{\beta}) \leq d_i$ for all $i \in [m]$. Now consider the
following sum:
\[ \sum_{ t=(t_1,t_2, \ldots, t_m) \in S} \left( \prod_{i=1}^m f_i(t_i) \right) p(t) .\]
Due to the assumption that $p(t)=0$ for all $t \in S$, this sum is $0$. On the other hand, we have the following way of rewriting the sum
\[ \sum_{ t=(t_1,t_2, \ldots, t_m) \in S} \left( \prod_{i=1}^m f_i(t_i) \right) p(t) = \sum_{\beta} p_{\beta} \left( \sum_{(t_1,t_2, \ldots, t_m) \in S} \prod_{i=1}^m f_i(t_i) t_i^{\beta_i} \right) .\]
 Note the following identity:
\[  \sum_{(t_1,t_2, \ldots, t_m) \in S} \prod_{i=1}^m f_i(t_i) t_i^{\beta_i}  = \prod_{i=1}^m  \sum_{t_i \in S_i} f_i(t_i) t_i^{\beta_i}  .\]
By the established properties of $f_i$, we have 
\[ \prod_{i=1}^m \left( \sum_{t_i \in S_i} f_i(t_i) t_i^{\beta_i} \right)  = 0 ,\]
\noindent unless $\beta=\alpha$. In conclusion, we have
\[ \sum_{ t=(t_1,t_2, \ldots, t_m) \in S} \left( \prod_{i=1}^m f_i(t_i) \right) p(t) = p_{\alpha} \left( \sum_{(t_1,t_2, \ldots, t_m) \in S} \prod_{i=1}^m f_i(t_i) t_i^{\alpha_i} \right) = p_{\alpha} .\]
By the theorem's hypothesis, $p_{\alpha} \neq 0$, which implies that there exists a $t \in S$ with $p(t) \neq 0$.
\end{proof}

\subsection{Proof of multivariate Schwartz-Zippel lemma}
\label{sec:proof-mSZ}
The proof involves detailed inequalities. For the sake of clarity in presentation, we only present the proof in the case $m=2$. For $m>2$, one simply repeats the proof we present below.
\subsubsection{Preparation for the proof}

\begin{lemma}
  \label{trivial}
  Let $\mathbb{F}$ be a characteristic zero field, $L$ be a collection of varieties in $\mathbb{F}^n$, and let $P$ be a collection of points in  $\mathbb{F}^n$. We denote the set of incidences between $L$ and $P$  with $\mathcal{I}(L,P)$. Suppose that for any $n$-tuple of distinct
  varieties $ V_1, V_2, \ldots, V_n$ in $L$ the B\'ezout bound holds:
  $ \abs{ \cap_{i=1}^n V_i } \leq d^n$. Then, we have
  \[ \abs{\mathcal{I}(L,P)} \leq 2 d \abs{L} \abs{P}^{1-\frac{1}{n}} + 2 n \abs{P} . \]
\end{lemma}
\begin{proof}
  For every $x \in P$ we define the set of incidences; $I_x := \{ V \in L : x \in V \}$. Then, we have

  \[  \abs{\mathcal{I}(L,P)} = \sum_{x \in P} \abs{I_x} \]

  \noindent We set the vector $I:= (\abs{I_x})_{x \in P}$. We have $\abs{\mathcal{I}(L,P)}= \norm{I}_1$. Now, we  consider the expansion of $\norm{I}_{n}^{n}$:

  \begin{equation} \label{ohnein}
    \norm{I}_{n}^{n} = \sum_{x \in P} \abs{I_x}^{n} = \sum_{x \in P} \abs{ \{ (V_1,V_2,\ldots,V_{n}) \in L \times \ldots \times L :  V_i \in I_x \; \text{for all} \; 1 \leq i \leq n \}  } 
  \end{equation}
  We divide (\ref{ohnein}) into two summands; the first summand consists of $n$-tuples $(V_1,\ldots,V_n)$ where all $V_i$ are distinct, and the second summand consists of $n$-tuples $V_i$ where at least one of the $V_i$ is repeated. This gives us the following:
  \begin{equation} \label{wasmachtsdu}
    \norm{I}_{n}^{n}  \leq \sum_{x \in P}  \abs{ \{ (V_1,V_2,\ldots,V_n) \in L \times \ldots \times L : V_i \neq V_j \; , \; V_i \in I_x \}  }  + n \norm{I}_{n-1}^{n-1}  
  \end{equation}
  Here we used the crude estimate $n \norm{I}_{n-1}^{n-1}$ to bound the second summand:
  \[ \sum_{x \in P} \abs{ \{ (V_1,V_2,\ldots,V_n) : V_i \in I_x , \; \text{at least one} \; V_i \; \text{is repeated} \}  }   \]
  Now we will do a double counting argument for the first summand in (\ref{wasmachtsdu}): instead of summing over $x \in P$ we sum over $n$-tuples $(V_1,V_2,\ldots,V_{n})$ where $V_i \in L$ are distinct. The B\'ezout assumption in the lemma statement gives us the following bound:
  \[ \norm{I}_{n}^{n}  \leq \binom{\abs{L}}{n} d^{n} +  n \norm{I}_{n-1}^{n-1} \]
  Note that $\norm{I}_{n-1}^{n-1}=\sum_{x \in P} \abs{I_x}^{n-1}$, so using H\"older's inequality for the pair $(\frac{n}{n-1},n)$ gives the following 
  \begin{equation} \label{nochmal}
    \norm{I}_{n}^{n}  \leq \binom{\abs{L}}{n} d^{n} + n\abs{P}^{\frac{1}{n}}  \norm{I}_{n}^{n-1} 
  \end{equation}
  \noindent Either we have $\norm{I}_{n} \leq 2 n \abs{P}^{\frac{1}{n}} $ or $\frac{1}{2} \norm{I}_{n} > n \abs{P}^{\frac{1}{n}}$. In the second case, we have 
  \begin{equation}
    \frac{1}{2} \norm{I}_{n}^{n}  \leq  \norm{I}_{n}^{n}  -  n \abs{P}^{\frac{1}{n}}  \norm{I}_{n}^{n-1} \leq  d^{n} \binom{\abs{L}}{n}
  \end{equation}
  \noindent where we used (\ref{nochmal}) for the latter inequality. Using H\"older's inequality and Stirling's estimate gives us the following:
  \[ \norm{I}_1^{n} \leq \norm{I}_{n}^n \abs{P}^{n-1} \leq  2 d^n \left( \frac{e \abs{L} }{n} \right)^n \abs{P}^{n-1} \leq 2 d^n \abs{L}^{n} \abs{P}^{n-1} \]
  \noindent In the first case, we would have $\norm{I}_1 \leq \norm{I}_{n} \abs{P}^{1-\frac{1}{n}} \leq 2 n \abs{P}$.
\end{proof}

\begin{lemma}
  \label{partition}
  Let $\mathbb{F}$ be a characteristic zero field, and let $S_1 \subset \mathbb{F}^{n_1}$ and
  $S_2 \subset \mathbb{F}^{n_2}$ be finite sets. Let $V$ be a variety in $\mathbb{F}^{n_1+n_2}$, suppose that for any two sets $U_1 , U_2$ with $\abs{U_i} > d^{n_i}$ the cartesian product $U_1 \times U_2$ is not included in $V$. Then there exist
  sets $S_{1i}$ with the following properties:
  \begin{enumerate}
  \item $S_1 = \cup_{j=1}^{t} S_{1j}$  with $t \leq d^{2n_1} $.
  \item For every  $1 \leq i \leq t$ and  any $I \subset S_{1i}$  with $\abs{I}=n_2$ we have the following: 
    \[
      \abs{ \cap_{x \in I} \{ y \in \mathbb{F}^{n_2} : (x,y) \in V \}  } \leq d^{n_2} .
    \]
  \end{enumerate}
\end{lemma}
\begin{proof}[Proof of Lemma~\ref{partition}]
  We consider all $n_2$-element subsets $I$ of $S_1$, and write down
  the corresponding varieties 
  \[ V_I:=\cap_{x \in I} \{ y \in \mathbb{F}^{n_2} : (x,y) \in V \} \]
  We discard all $V_I$ that have less then $d^{n_2}+1$ many elements, and keep track all the rest. Suppose we have a list
  $V_1,V_2,\ldots,V_M \subset \mathbb{C}^{n_2}$. For every $V_i$ in the list, 
  we define $U_i \subset S_1$ as follows:
  \[ U_i := \{ x \in S_1 : (x,y) \in V \; \text{for all} \; y \in V_i \}
    .\] Since $\abs{V_i} \geq d^{n_2}+1$, we must have
  $\abs{U_i} \leq d^{n_1}$. Assume otherwise, w.l.o.g. say
  $\abs{U_1}>d^{n_1}$ and pick a subset $W_1$ of $V_1$ with more than
  $d^{n_2}$ many elements. Then $p$ vanishes on $U_1 \times W_1$ by
  construction, and this gives a contradiction by Lemma \ref{witness}.

  The rest of the proof is as follows: We will create a poset out of
  $U_1,U_2,\ldots,U_M$, this poset will have at most $d^{n_1}$ many
  layers and we color the $U_i$ in every layer with $d^{n_1}$ many
  colors. Thus, in total, we use at most $d^{2n_1}$ colors and these
  colors will correspond to $S_{1j}$.

  We define the poset as follows: we use the partial order $\succ$
  given by the inclusion of the varieties $V_i$. For instance, if
  $V_i \cup V_j \subseteq V_k$, then we have $U_k \succ U_i$ and
  $U_k \succ U_j$. Also note that here we must have $U_k \neq U_j$ and  $U_k \neq U_i$, otherwise either $V_i$ or $V_j$ is redundant and  will be discarded from our list of varieties together with the corresponding $U_i$ (or $U_j$).

  First, we observe that the longest chain in this poset could be of
  length $d^{n_1}$: $U_k \succ U_i$ implies $\abs{U_i} > \abs{U_k}$ for $i \neq k$,   and we know $\abs{U_i} \leq d^{n_1}$ for all $U_i$. Thus we have a poset with at most $d^{n_1}$ many layers. We start with the first layer and color elements of every $U_\ell$ in this layer using $d^{n_1}$ many colors. Suppose for some $U_i$ and $U_j$ in the first layer, we   have $U_i \cap U_j \neq \emptyset$. Then the corresponding variety   $V_i \cup V_j$ (or a variety including this variety) must be in our   list of varieties, say $V_k=V_i \cup V_j$,   where we have $U_i \cap U_j \subseteq U_k$. By definition of the poset, we have $U_k \succ U_i$ and $U_k \succ U_j$, and that the set $U_k$ is in the second layer or higher in our poset. The elements in $U_i \cap U_j \subset U_k$ will be re-colored with new colors that are going to be used in the second layer of the poset. Similarly, if $U_i \cap U_j \cap  U_{\ell} \neq \emptyset$  for some $i,j,\ell$, the corresponding set $U_i \cap U_j \cap  U_{\ell}  \subset U_k$  satisfies $U_k \succ U_i \cap U_j \succ U_i$ and is at least in the third layer of the poset etc..  Hence, using $d^{n_1}$ different colors at every layer we can guarantee that for any two elements $u,v \in U_{\ell}$ for some $\ell$, the color of $u$ and the color $v$ are different. Thus, at most $d^{2n_1}$ many colors suffice. 
\end{proof}

\subsubsection{$\lambda$-irreducibility for real varieties}
\begin{definition}[Real $\lambda$-irreducible variety]
Let $\lambda=(\lambda_1,\lambda_2,\ldots,\lambda_m)$ be a vector with positive integer coordinates, let $n=\lambda_1+\ldots+\lambda_m$, and let $V \subseteq \mathbb{R}^{\lambda_1} \times \mathbb{R}^{\lambda_2} \ldots \times\mathbb{R}^{\lambda_m}$ be a real algebraic set. We say $V$ is $\lambda$-reducible if there exist positive dimensional real varieties $C_1, C_2,\ldots,C_m$ with 
\[ C_1 \times C_2 \times \ldots \times C_m \subseteq V. \]
We say $V$ is $\lambda$-irreducible otherwise.
\end{definition}

There is a translation between the complex and the real definitions as follows.
\begin{proposition}
 Let $\lambda=(\lambda_1,\lambda_2,\ldots,\lambda_m)$ be a vector with positive integer coordinates, let $n=\lambda_1+\ldots+\lambda_m$, and let $W \subset \mathbb{C}^n$ be variety, and denote the embedding of $W$ into $\mathbb{R}^{2n}$ with $V$. We set $2 \lambda$ as $2\lambda:=(2\lambda_1, 2\lambda_2,\ldots, 2\lambda_m)$. Then, $W$ is a complex $\lambda$-irreducible variety if and only if $V$ is $2\lambda$-irreducible real algebraic variety.  
\end{proposition}
\begin{proof}
$(\Rightarrow)$ Assume on the contrary that $V$ is $2\lambda$-reducible, that is, there exist real varieties $C_i\subset\RR^{2\lambda_i}$ of dimension $\geq 1$ such that \[
C_1\times C_2\times\dots\times C_m\subset V.
\] As $W$ is a complex variety, it is Zariski closed and hence \[
\overline{C_1\times C_2\times\dots\times C_m}\subset W.
\] It is standard to show that \[
\overline{C_1\times C_2\times\dots\times C_m} = \overline{C_1}\times\overline{C_2}\times\dots\times\overline{C_m}.
\] Moreover, $\dim_\CC \overline{C_i}\geq 1$; thus, $W$ is complex $\lambda$-reducible which is a contradiction.

$(\Leftarrow)$ Say \[
C_1\times C_2\times\dots \times C_m\subset W.
\] The standard embedding of $C_i$ into $\RR^{2\lambda_i}$ are real algebraic and has dimension \\
$2\dim_{\mathbb{C}} C_i$. Hence, $V$ contains the product $\prod_{i=1}^m C_i$ of real algebraic sets of dimension at least $1$, which contradicts with the assumption that $V$ is $2\lambda$-irreducible. 
\end{proof}

Moreover, one can also bound the degree of the $V$ using $\deg W$. 
\begin{lemma} \label{transfer}
Let $W \subset \mathbb{C}^n$ be a variety of degree $d$. Then, the standard embedding of $V \subset \mathbb{R}^{2n}$ is a real algebraic variety. Moreover, the complexification $\tilde{V} \subset \mathbb{C}^{2n}$ of the real algebraic variety $V \subset \mathbb{R}^{2n}$ has degree at most $O(d)^{2n}$.
\end{lemma}
\begin{proof}
The degree $d$ complex variety $W$ is the zero locus of some complex polynomials $f_i$ of degree at most $d$ (see, for example, \cite[Prop.~3.5]{EHV-primdec-92}). Each $f_i$ splits into the real and the imaginary part, say $p_i,q_i$ respectively. Setting \[
F = \sum_{i=1}^n p_i^2+q_i^2
\] we have\[
V = Z(\sum_{i=1}^n p_i^2+q_i^2)\subset\RR^{2n},
\] as a real algebraic set. Now we can use the result of \cite{roy}: The degree of the complexification of a real variety $Z(F)\subset\RR^n$ is bounded by $O(\deg F)^{2n}$. In particular, \[
\deg\tilde{V} \leq O(d)^{2n}. 
\]
\end{proof}

\subsubsection{Proof for the case $m=2$}

\begin{proposition}
  \label{SZJJ}
Let $n=n_1+n_2$ be a two partition, and let $S_1 \subset \mathbb{R}^{n_1}$, $S_2 \subset \mathbb{R}^{n_2}$ be finite sets. Let $V \subset \mathbb{R}^n$ be a $k$-dimensional irreducible real algebraic variety where the degree of the complexification of $V$ is $d$ and $d \geq 2$.  Suppose that for every $n_2$-element subset $I$ of $S_1$, and for every $n_1$-element subset $J$ of $S_2$, we have the following
\[ \abs{ \cap_{x \in I} \{ y : (x,y) \in V \}  } \leq d^{n_2} \; , \; \abs{ \cap_{y \in J} \{ x : (x,y) \in V \} } \leq d^{n_1}. \]
Then for every $\varepsilon >0$, we have
\[ \abs{V \bigcap S_1 \times S_2} \leq c_{\varepsilon} d^{n^{3k}+1} \abs{S_1}^{1-\frac{1}{n_1+1}+\varepsilon} \abs{S_2}^{1-\frac{1}{n_2+1} + \varepsilon} + 2n d^{n^3}( \abs{S_1} + \abs{S_2}) ,\]
where $c_{\varepsilon}$ is a constant that depends only on $\varepsilon$.
\end{proposition}

\begin{proof}[Proof of Proposition~\ref{SZJJ}]
The proof will be by double-induction on the dimension of real algebraic set and on the size $\abs{S_1} + \abs{S_2}$. 

For a zero dimensional real algebraic set, the result immediately follows from the assumption on the degree of the complexification.  This is the base for induction on the dimension.

For the base of induction on $\abs{S_1} + \abs{S_2}$:
By Lemma~\ref{trivial} we have the following direct bound 
\[ \abs{V \bigcap S_1 \times S_2} \leq 2d\abs{S_1} \abs{S_2}^{1-\frac{1}{n_2}} + 2n_2 \abs{S_2} . \]
If $\abs{S_{1}}^{\frac{1}{n_1+1}} \leq \abs{S_2}^{\frac{1}{n_2}-\frac{1}{n_2+1}}$, then  we have
\begin{equation} \label{1}
 \abs{V \bigcap S_1 \times S_2} \leq 2d \abs{S_1}^{1-\frac{1}{n_1+1}} \abs{S_2}^{1-\frac{1}{n_2+1}} + 2n_2 \abs{S_2} .
\end{equation}
Note that the inequality (\ref{1}) gives a bound stronger than our claim, and there is no work to do in this special case. This establishes the claim whenever one of the two sets are sufficiently small, and this gives the basis to start the induction. From this point on, we assume  $\abs{S_{1}}^{\frac{1}{n_1+1}} \geq \abs{S_2}^{\frac{1}{n_2}-\frac{1}{n_2+1}}$ and  $\abs{S_2}^{\frac{1}{n_2+1}} \geq \abs{S_1}^{\frac{1}{n_1}-\frac{1}{n_1+1}}$. Also, for $n_1=1$ or $n_2=1$ the claim holds immediately: suppose $n_2=1$, then for every point $x \in S_1$ there are at most $d$ corresponding $y \in S_2$ with $(x,y) \in V$ which gives $d \abs{S_1}$ upper bound. So, we assume $\min \{n_1 , n_2 \} \geq 2$.

The  proof below will count every point on $V \bigcap S_1 \times S_2$ twice. We use Lemma ~\ref{zahl} to find  degree $d^2$ polynomials $h_1$ (respectively $h_2$) gives the following partitioning of $\mathbb{R}^{n_1}$ (respectively $\mathbb{R}^{n_2}$):  
\[ \mathbb{R}^{n_1} =  Z(h_1) \cup \Omega_1 \cup \Omega_2 \cup \ldots \cup \Omega_T ,\]
where $T \leq d^{2n_1}$ and $\abs{\Omega_i \cap S_1} \leq \abs{S_1} / d^{2n_1}$ for all $i=1,2,\ldots, T$, we also require that $h_1h_2$ does not vanish on $V$.  Note that to guarantee $h_1 h_2$ does not vanish on $V$ is simple thanks to Lemma \ref{zahl}: pick $(x,y) \in V$, define $V_y := \{ x \in \mathbb{R}^{n_1} : (x,y) \in V \}$, and use Lemma \ref{zahl} to make sure $h_1$ does not vanish on $V_y$. 

Since we are only interested in counting zeros on $S_1 \times S_2$, in the rest of the proof $\Omega_i$ would simply denote $\Omega_i \cap S_1$. Now we would like to count zeros of $p$ on $\left( S_1 \setminus Z(h_1) \right) \times S_2$. For this we define the following sets:
\[ L_i := \{ y \in S_2 : \exists x \in \Omega_i \; \text{such that} \; (x,y) \in V \} .\]
Note that for a fixed $y \in S_2$ the set $\{ x \in \mathbb{R}^{n_1} :  (x,y) \in V \}$ is a variety of dimension at most $n_1-1$. Therefore the number of connected components of the semialgebraic set
\[ \{ x\in \mathbb{R}^{n_1} : (x, y) \in V  \; , \; h_1(x) \neq 0 \} \] 
is bounded by $O(d^{2n_1-2})$ due to Theorem \ref{SS}. This simply shows that any $y \in S_2$ can be included in at  most $O(d^{2n_1-2})$ many $L_i$'s. So we have
\begin{equation} \label{nice}
 \sum_{i=1}^{T} \abs{L_i} \leq d^{2n_1-2} \abs{S_2}. 
\end{equation}

Using the induction hypothesis, we can bound incidences between $\Omega_i$ and $L_i$:
\[ \abs{  \Omega_i \times L_i \bigcap Z(p)  }   \leq   c_{\varepsilon} d^{n^{3k}+1} \abs{\Omega_i}^{1-\frac{1}{n_1+1}+\varepsilon} \abs{L_i}^{1-\frac{1}{n_2+1}+ \varepsilon} + 2n d^{n^3} \left(  \abs{L_i} + \abs{\Omega_i} \right) .\]
Summing through $\Omega_i$ we have:

\begin{multline*}
\abs{ (S_1 - Z(h_1)) \times S_2  \bigcap Z(p)} \leq c_{\varepsilon} d^{n^{3k}}\sum_{i=1}^T d \abs{\Omega_i}^{1-\frac{1}{n_1+1}+\varepsilon} \abs{L_i}^{1-\frac{1}{n_2+1}+ \varepsilon} +  
\\  2n d^{n^3} \sum_{i=1}^T ( \abs{L_i} +  \abs{\Omega_i} ) .
\end{multline*}
 
Since $\abs{\Omega_i} \leq \abs{S_1} / d^{4n_1}$ for all $i$,  we have

\[
 \sum_{i=1}^T  d \abs{\Omega_i}^{1-\frac{1}{n_1+1}+\varepsilon} \abs{L_i}^{1-\frac{1}{n_2+1} + \varepsilon}  \leq d^{1- 4n_1(1-\frac{1}{n_1+1}+\varepsilon)} \abs{S_1}^{1-\frac{1}{n_1+1}+\varepsilon} \sum_{i=1}^T \abs{L_i}^{1-\frac{1}{n_2+1} + \varepsilon} .
 \]
Using H\"older's inequality and (\ref{nice}) we have
\[
  \begin{aligned}
  \sum_{i=1}^{T} \abs{L_i}^{1-\frac{1}{n_2+1}+\varepsilon}
  & \leq T^{\frac{1}{n_2+1}-\varepsilon} \left( \sum_{i=1}^{T} \abs{L_i} \right)^{1-\frac{1}{n_2+1}+\varepsilon} \\
  &  \leq d^{ 4n_1 (\frac{1}{n_2+1}-\varepsilon) + (2n_1-2)(1-\frac{1}{n_2+1}+\varepsilon)} \abs{S_2}^{1-\frac{1}{n_2+1}+\varepsilon}.
  \end{aligned}
\]
Now we collect the exponents of $d$ in last two inequalities and make it human readable:
\[
  4n_1 (\frac{1}{n_2+1}-\varepsilon) +
  (2n_1-2)(1-\frac{1}{n_2+1}+\varepsilon) +
  1- 4n_1(1-\frac{1}{n_1+1}+\varepsilon) .
\]
We can bound this expression, using basic algebra and the inequality
$\min \{ n_1 , n_2 \} \geq 2$, as follows:
\[
\begin{aligned}
  (2n_1+2)(\frac{1}{n_2+1}-\varepsilon) + 2n_1 -2 -4n_1 - 4n_1 \varepsilon + 5
  & \leq \frac{11 - 4n_1}{3} -4n_1 \varepsilon \\
  & \leq 1 -\frac{4(n_1-2)}{3} -4n_1 \varepsilon .
\end{aligned}
\]
All in all, we have
\begin{multline*}
  \abs{\left( S_1 \setminus Z(h_1) \right) \times S_2  \bigcap Z(p)}
  \\  \leq c_{\varepsilon} d^{1-\frac{4(n_1-2)}{3}-4n_1\varepsilon}
  \abs{S_1}^{1-\frac{1}{n_1+1}+\varepsilon} \abs{S_2}^{1-\frac{1}{n_2+1}+\varepsilon}
   +  2n d^{n^3} \left( \abs{S_1} + d^{2n_1-2} \abs{S_2}  \right) .
\end{multline*}

We repeat the same counting argument for the zeros of $p$ on
$S_1 \times \left( S_2 \setminus Z(h_2) \right)$, which gives us the
following the following upper bound:
\begin{multline*}
  \abs{ S_1 \times \left( S_2 \setminus Z(h_2) \right) \bigcap Z(p) }
  \\ \leq  c_{\varepsilon} d^{ 1-\frac{4(n_2-2)}{3}-4n_2\varepsilon} \abs{S_1}^{1-\frac{1}{n_1+1}+\varepsilon}
  \abs{S_2}^{1-\frac{1}{n_2+1}+\varepsilon} +  2n d^{n^3} \left( d^{2n_2-2} \abs{S_1} +  \abs{S_2}  \right) .
\end{multline*}
We note that the difference 
\[ d^{n^{3k}}(d - d^{1-\frac{4(n-4)}{3}-4n \varepsilon}) \abs{S_1}^{1-\frac{1}{n_1+1}+\varepsilon} \abs{S_2}^{1-\frac{1}{n_2+1}+\varepsilon} \]
is bigger than $d^{n^3(n-4)} (\abs{S_1} + \abs{S_2})$, so for this part of the induction we are done.

Now we are interested in counting  $\left( S_1 \cap Z(h_1) \right) \times \left( Z(h_2) \cap S_2 \right) \bigcap V$. Let $h=h_1h_2$, and let $\overline{Z(h)}$ be the complexification of the real zero set $Z(h)$. Let $\overline{V}$ denote the complexification of $V$. And let $W$ denote the real part of $\overline{Z(h)} \cap \overline{V}$.  By construction, $Z(h_1) \times Z(h_2) \bigcap V \subseteq W$. We denote the real ideals of $V$ and $W$ with $I(V)$ and $I(W)$, and the real quotient rings with $\mathcal{P}(V)$ and $\mathcal{P}(W)$. $V$ is irreducible over the reals, so we have $\dim \mathcal{P}(V)= \dim V$. Since $h$ is not included in $I(V)$, we have $\dim \mathcal{P}(W) < \dim \mathcal{P}(V)=\dim(V)$. Therefore, the dimension of $W$ is at most one less than the dimension of $V$. We also know that $W$ is $(n_1,n_2)$-irreducible: assume $W$ includes a cartesian products of real curves, this implies that the cartesian product of real curves is included in $\overline{V}$, and hence the cartesian product of real curves is included in $V$, which gives a contradiction. The degree of $h$ is $d^4$ and the degree of $\overline{V}$ is $d$, so the degree of the complexification of $W$ is bounded by $d^5$. By using the induction hypothesis on the dimension of the real algebraic set, we can bound the incidences between an irreducible component of $W$ and $S_1 \times S_2$ as:
\begin{equation} \label{parse}
  c_{\varepsilon} d^{5n^{3k-3}+5} \abs{S_1}^{1-\frac{1}{n_1+1}+\varepsilon} \abs{S_2}^{1-\frac{1}{n_2+1} + \varepsilon} + 2(n-2) d^{(n-2)^3} \left(  \abs{S_1} + \abs{S_2} \right).
\end{equation}
The number of connected components of $W$ is bounded by $(2d^5)^{n}$ as follows:  we can write-down $W$ using polynomials that have degree at most  $d^5$ (see e.g. Lemma 4.2 of \cite{tao-solymosi}), using classical Thom-Milnor bound then yields the estimate $(2d^5)^{n}$. To close the induction all we need is to have $(2d^5)^n d^{5n^{3k-3}+5} \leq d^{n^{3k}+1}$, which holds since $n \geq 4$ and $d \geq 2$. 
\end{proof}

Now we are ready to state and prove Theorem~\ref{thm-SZ} for $m = 2$.

\begin{theorem}[The case $m=2$] \label{SZjunior}
Let $S_1 \subset \mathbb{C}^{n_1}$ and $S_2 \subset \mathbb{C}^{n_2}$ be finite sets. Let $n=n_1+n_2$, and let $p$ be a $(n_1,n_2)$-irreducible polynomial of degree $d \geq 2$. Then, for every $\varepsilon >0$ we have 
 \[ \abs{Z(p) \bigcap S_1 \times S_2} = O_{n,d,\varepsilon} \left(  \abs{S_1}^{1-\frac{1}{n_1+1}+\varepsilon} \abs{S_2}^{1-\frac{1}{n_2+1} + \varepsilon} +   \abs{S_1} +  \abs{S_2}  \right)  ,\]
 where $O_{n,d,\varepsilon}$ only hides constants depending on $\varepsilon,d$ and $n$.
\end{theorem}
\begin{proof}
The proof is based on Lemma \ref{witness}, Lemma \ref{partition}, Lemma \ref{transfer} and  Proposition~\ref{SZJJ}.  Suppose two sets $S_1$ and $S_2$ are given and we use Lemma \ref{witness} and  Lemma \ref{partition} to create partitions $S_1 = \cup_{i=1}^M S_{1i}$ and $S_2=\cup_{j=1}^N S_{2j}$ where $M \leq d^{2n_1}$, $N \leq d^{2n_2}$. We embed the sets $S_{1i}$ into $\mathbb{R}^{2n_1}$ and $S_{2j}$ into $\mathbb{R}^{2n_2}$. By Lemma \ref{transfer} we know that $Z(p) \subset \mathbb{R}^{2n_1+2n_2}$ is a $(2n_1,2n_2)$-irreducible real variety, and the complexification of $Z(p)$ in $\mathbb{C}^{2n_1+2n_2}$ has degree at most $d^{2n}$. To apply Proposition \ref{SZJJ}, we need $p$ to be irreducible but this not an issue since $p$ can have at most $d$ many irreducible components. So, w.l.o.g. we treat $Z(p) \subset \mathbb{R}^{2n}$ as an irreducible real variety and apply Proposition \ref{SZJJ} for all pairs of $S_{1i},S_{2j}$. This gives the following bound:
\[  \abs{V \cap S_1 \times S_2}  =  O_{d,n,\varepsilon} \left( \sum_{i,j}  \abs{S_{1i}}^{1-\frac{1}{1+n_1}+\varepsilon} \abs{S_{2j}}^{1-\frac{1}{1+n_2}+\varepsilon} +  \abs{S_{1i}} \abs{S_{2j}} \right) \]

By H\"older's inequality we have
\[
  \sum\nolimits_{i=1}^M \abs{S_{1i}}^{1-\frac{1}{1+n_1}+\varepsilon} \leq M^{\frac{1}{1+n_1}} \abs{S_1}^{1-\frac{1}{1+n_1}+\varepsilon} \leq d^2  \abs{S_1}^{1-\frac{1}{1+n_1}+\varepsilon} .
\]
%
\text{Similarly} \quad
$ \sum_{i=1}^N \abs{S_{2j}}^{1-\frac{1}{1+n_2}+\varepsilon} \leq N^{\frac{1}{1+n_2}} \abs{S_1}^{1-\frac{1}{1+n_2}+\varepsilon} \leq d^2  \abs{S_2}^{1-\frac{1}{1+n_2}+\varepsilon} .$
Note that 
\[
  \sum_{i,j} \abs{S_{1i}}^{1-\frac{1}{1+n_1}+\varepsilon} \abs{S_{2j}}^{1-\frac{1}{1+n_2}+\varepsilon} =  \left( \sum_{i=1}^M \abs{S_{1i}}^{1-\frac{1}{1+n_1}+\varepsilon} \right) \left( \sum_{i=1}^N \abs{S_{2j}}^{1-\frac{1}{1+n_2}+\varepsilon} \right) .\]
Also note that, since $1 \leq i \leq M \leq d^{2n_1}$ and $1 \leq j \leq N \leq d^{2n_2}$, we have
\[ \sum_{1 \leq i \leq M} \sum_{1 \leq j \leq M}  \abs{S_{1i}} +  \abs{S_{2j}}  \leq d^{2n} (\abs{S_1} + \abs{S_2}) ,\]
which completes the proof of Theorem \ref{SZjunior}. 
\end{proof}

\section*{Acknowledgements}
Our interest in the topic was sparked by reading the survey article \cite{tao},
and the influence of this survey can probably be seen throughout this note. We
are grateful to Prof. Terence Tao for his splendid exposition. Many thanks go to
\"Ozg\"ur Ki\c{s}isel for enjoyable discussions at the initial stages of this
work, to Pravesh Kothari for allowing A.E. to forget about computers sometimes,
and to J. Maurice Rojas for supporting A.E. and J.M. in his REU program at Texas
A\&M Mathematics. We also thank Frank de Zeeuw, Adam Sheffer, D\"om\"ot\"or
P{\'a}lv\"olgyi for useful remarks, and to Joshua Zahl for a wonderful question
that helped us to locate a mistake in one of our proofs in the earlier version
of this paper. Last but not least, we are grateful to the two anonymous referees
for their very helpful comments and remarks.

\bibliographystyle{siamplain}
\bibliography{mSZ}

\begin{thebibliography}{10}

\bibitem{alon}
{\sc N.~Alon}, {\em {Combinatorial Nullstellensatz}}, Combinatorics,
  Probability and Computing, 8 (1999), pp.~7--29.

\bibitem{barone-basu}
{\sc S.~Barone and S.~Basu}, {\em Refined bounds on the number of connected
  components of sign conditions on a variety}, Discrete \& Computational
  Geometry, 47 (2012), pp.~577--597.

\bibitem{Canny-GCP-90}
{\sc J.~Canny}, {\em Generalised characteristic polynomials}, Journal of
  Symbolic Computation, 9 (1990), pp.~241--250.

\bibitem{dvir2009size}
{\sc Z.~Dvir}, {\em {On the size of Kakeya sets in finite fields}}, Journal of
  the American Mathematical Society, 22 (2009), pp.~1093--1097.

\bibitem{dvir-survey}
{\sc Z.~Dvir}, {\em Incidence theorems and their applications}, Foundations and
  Trends{\textregistered} in Theoretical Computer Science, 6 (2012),
  pp.~257--393.

\bibitem{EHV-primdec-92}
{\sc D.~Eisenbud, C.~Huneke, and W.~Vasconcelos}, {\em Direct methods for
  primary decomposition}, Inventiones mathematicae, 110 (1992), pp.~207--235.

\bibitem{EmiMou-met-99}
{\sc I.~Z. Emiris and B.~Mourrain}, {\em Matrices in elimination theory},
  Journal of Symbolic Computation, 28 (1999), pp.~3--44.

\bibitem{adam-semi}
{\sc J.~Fox, J.~Pach, A.~Sheffer, A.~Suk, and J.~Zahl}, {\em {A semi-algebraic
  version of Zarankiewicz's problem}}, Journal of the European Mathematical
  Society, 19 (2017), pp.~1785--1810.

\bibitem{fox}
{\sc J.~Fox, J.~Pach, A.~Sheffer, A.~Suk, and J.~Zahl}, {\em A semi-algebraic
  version of zarankiewicz's problem}, Journal of the European Mathematical
  Society, 19 (2017), pp.~1785--1810.

\bibitem{GKZ}
{\sc I.~M. Gelfand, M.~Kapranov, and A.~Zelevinsky}, {\em Discriminants,
  resultants, and multidimensional determinants}, Springer Science \& Business
  Media, 2008.

\bibitem{guth2016polynomial}
{\sc L.~Guth}, {\em Polynomial methods in combinatorics}, vol.~64, American
  Mathematical Soc., 2016.

\bibitem{guth2010algebraic}
{\sc L.~Guth and N.~H. Katz}, {\em {Algebraic methods in discrete analogs of
  the Kakeya problem}}, Advances in Mathematics, 225 (2010), pp.~2828--2839.

\bibitem{guth-katz}
{\sc L.~Guth and N.~H. Katz}, {\em {On the Erd{\H{o}}s distinct distances
  problem in the plane}}, Annals of Mathematics,  (2015), pp.~155--190.

\bibitem{lipton}
{\sc R.~J. Lipton}, {\em {The curious history of Schwarz-Zippel Lemma}},
  \url{https://rjlipton.wordpress.com/2009/11/30/the-curious-history-of-the-schwartz-zippel-lemma/},
   (2009).

\bibitem{dezeeuw}
{\sc H.~N. Mojarrad, T.~Pham, C.~Valculescu, and F.~de~Zeeuw}, {\em
  Schwartz-zippel bounds for two-dimensional products}, Discrete Analysis,
  (2017), p.~2750.

\bibitem{Mumford-AG-I-95}
{\sc D.~Mumford}, {\em Algebraic geometry I: complex projective varieties},
  Springer Science \& Business Media, 1995.
\newblock Reprint of the 1976 edition.

\bibitem{pach-sharir}
{\sc J.~Pach and M.~Sharir}, {\em On the number of incidences between points
  and curves}, Combinatorics, Probability and Computing, 7 (1998),
  pp.~121--127.

\bibitem{sharir}
{\sc O.~E. Raz, M.~Sharir, and J.~Solymosi}, {\em "{P}olynomials vanishing on
  grids: The {E}lekes-{R}{\'o}nyai problem revisited"}, in Proceedings of the
  thirtieth annual symposium on Computational geometry, ACM, 2014, p.~251.

\bibitem{roy}
{\sc M.-F. Roy and N.~Vorobjov}, {\em The complexification and degree of a
  semi-algebraic set}, Mathematische Zeitschrift, 239 (2002).

\bibitem{saxena}
{\sc N.~Saxena}, {\em Progress on polynomial identity testing.}, Bulletin of
  the EATCS, 99 (2009), pp.~49--79.

\bibitem{schmidt}
{\sc J.~Schmid}, {\em {On the affine B\'ezout inequality}}, Manuscripta
  Mathematica, 88 (1995), pp.~225--232.

\bibitem{Sheffer-book-09}
{\sc A.~Sheffer}, {\em Polynomial Methods and Incidence Theory}, (Book in
  progress), 2019.
\newblock \url{http://faculty.baruch.cuny.edu/ASheffer/000book.pdf}.

\bibitem{sheffer}
{\sc A.~Sheffer, E.~Szab{\'o}, and J.~Zahl}, {\em Point-curve incidences in the
  complex plane}, Combinatorica, 38 (2018), pp.~487--499.

\bibitem{tao-solymosi}
{\sc J.~Solymosi and T.~Tao}, {\em An incidence theorem in higher dimensions},
  Discrete \& Computational Geometry, 48 (2012), pp.~255--280.

\bibitem{spencer}
{\sc J.~Spencer, E.~Szemer{\'e}di, and W.~T. Trotter}, {\em Unit distances in
  the euclidean plane}, in Graph theory and combinatorics, Academic Press,
  1984.

\bibitem{szemeredi-trotter}
{\sc E.~Szemer{\'e}di and W.~T. Trotter}, {\em Extremal problems in discrete
  geometry}, Combinatorica, 3 (1983), pp.~381--392.

\bibitem{Tao-bezout}
{\sc T.~Tao}, {\em Bezout’s inequality},
  \url{https://terrytao.wordpress.com/2011/03/23/bezouts-inequality/}.

\bibitem{tao}
{\sc T.~Tao}, {\em Algebraic combinatorial geometry: the polynomial method in
  arithmetic combinatorics, incidence combinatorics, and number theory}, EMS
  Surveys in Mathematical Sciences, 1 (2014), pp.~1--46.

\bibitem{toth}
{\sc C.~D. T{\'o}th}, {\em {The Szemer{\'e}di-Trotter theorem in the complex
  plane}}, Combinatorica, 35 (2015), pp.~95--126.

\bibitem{VogPat-rit-83}
{\sc W.~Vogel and D.~P. Patil}, {\em Remarks on the algebraic approach to
  intersection theory.}, Monatshefte f\"ur Mathematik, 96 (1983), pp.~233--250,
  \url{http://eudml.org/doc/178150}.

\end{thebibliography}
\end{document}